\documentclass[11pt]{amsart}

\usepackage{setspace} 
\usepackage{amssymb}
\usepackage[dvips]{color}
\usepackage{amsmath}
\usepackage{amsthm}
\usepackage[a4paper]{geometry}
\usepackage{amsfonts}
\usepackage[all]{xypic} 
\usepackage{bbm}
\usepackage{xcolor}
\usepackage[normalem]{ulem}

\usepackage{color}
\usepackage[colorlinks=true,linkcolor=blue,citecolor=blue,pdfpagelabels=false]{hyperref}
\usepackage{cleveref}
\usepackage{url}

\usepackage{hyperref}

\usepackage[colorinlistoftodos]{todonotes}

\allowdisplaybreaks[4] 
\newtheoremstyle{myremark}     {10pt}{10pt}{}{}{\bfseries}{.}{.5em}{}

\newtheorem{thm}{Theorem}[section]
\newtheorem{cor}[thm]{Corollary}
\newtheorem{lem}[thm]{Lemma}
\newtheorem{prop}[thm]{Proposition}
\theoremstyle{definition}
\newtheorem{defn}[thm]{Definition}

\theoremstyle{myremark}
\newtheorem{rem}[thm]{Remark}
\numberwithin{equation}{section}

\def\A{\mathbb A}
\def\B{\mathbb B}
\def\C{\mathbb C}

\def\N{\mathbb N}

\def\R{\mathbb R}

\def\S{\mathbb S}

\newcommand{\CA}{\mathcal{A}}
\newcommand{\CB}{\mathcal{B}}
\newcommand{\CC}{\mathcal{C}}

\newcommand{\CE}{\mathcal{E}}
\newcommand{\CF}{\mathcal{F}}

\newcommand{\CH}{\mathcal{H}}

\newcommand{\CK}{\mathcal{K}}

\newcommand{\CM}{\mathcal{M}}

\newcommand{\CP}{\mathcal{P}}

\newcommand{\CS}{\mathcal{S}}

\newcommand{\CU}{\mathcal{U}}

\newcommand{\abs}[1]{\left\vert#1\right\vert}

\newcommand{\norm}[1]{\left\Vert#1\right\Vert}

\newcommand{\Aut}{\operatorname{Aut}}

\begin{document}
	
	\title[Noncommutative Wiener-Wintner Theorems]{ Non-Commutative Wiener-Wintner theorem for amenable group actions}
	
	\author[P. Bikram]{Panchugopal Bikram$^{1, 2}$}
\address{$^{1}$School of Mathematical Sciences,
National Institute of Science Education and Research,  Bhubaneswar, Odisha, India  \newline  $^{2}$Homi Bhabha National Institute, Training School Complex, Anushakti Nagar, Mumbai 400094, India}
	
\author[S.Kundu]{Sudipta Kundu$^{1, 2}$}
	
\author[Hariharan G.]{Hariharan G$^{1, 2}$}
	
	\email{bikram@niser.ac.in}  \email{sudipta.kundu@niser.ac.in}
    	\email{hariharan.g@niser.ac.in} 
    \maketitle

\begin{abstract}
Let \(G\) be a locally compact, second countable, amenable group acting on a finite von Neumann algebra 
\((\mathcal{M},\tau)\) by trace-preserving automorphisms. In this article, we establish a Jacobs-de Leeuw-Glicksberg decomposition for this action, yielding a decomposition of \(\mathcal{M}\) into its almost periodic and weakly mixing components. We also prove a noncommutative version of the van der Corput lemma. As an application, we establish a noncommutative Wiener-Wintner theorem for amenable group actions on finite von Neumann algebras.

\end{abstract}

\section{\textbf{Introduction}}

In this article, we study a Wiener-Wintner theorem using a Jacobs-de Leeuw-Glicksberg   (JdLG) decomposition of a finite von Neumann algebra $\CM$ acted upon continuously, and ergodically, by a locally compact second countable (l.c.s.c.) amenable group $G$. Wiener-Wintner theorems are a class of pointwise convergence results for weighted ergodic averages on a full measure set independent of the weight. The study of Wiener-Wintner theorems was first started by Wiener and Wintner in their 1941 paper \cite{classicalZWW} for the case of a single measure preserving transformation $T$ acting on a measure space $X$,  the theorem states that given an ergodic measure preserving system $(\Omega,\mathcal{B},\mu,T)$ and a complex number (referred to as weight) $\lambda$ of modulus 1, for any $f\in L^1(\Omega,\mu)$ there is a full measure set $\Omega_f$ independent of $\lambda$ such that the average of the form (hereafter referred to as a Wiener-Wintner average)
\begin{align*}a_n(f,\lambda):=\frac{1}{n}\displaystyle\sum_{k=0}^{n-1}\lambda^kf(T^kx)
\end{align*}
converges for all $x\in \Omega_f$. Further advances were made by Bellow and Losert in their paper \cite{BellowLosert}, where pointwise a.e. convergence was proved for different weight classes and for subsequential counterparts of the Wiener-Wintner averages in the measure space setting.

\noindent A noncommutative counterpart of this theorem in the setting of a von Neumann algebra $\CM$, equipped with an ergodic automorphism $\alpha \in Aut(\CM)$ was studied by Litvinov in his 2014 work, \cite{Lit2014}, where the uniform convergence in the bilaterally almost uniform (b.a.u.) sense, of Wiener-Wintner averages of the form
\begin{align*}
    \frac{1}{N}\sum_{i=0}^{N-1}\lambda^{k}\alpha^{k}(x), \lambda \in \C^{1}
\end{align*}
is proved. Similar results were studied for different weight classes in the multiparameter case by Hong \cite{Hong-Sun-2018} and later, by \cite{OBrien2021NoncommutativeWT}.\\
Parallel to these developments, a Wiener-Wintner theorem for the case of a continuous and ergodic action of a l.c.s.c. amenable group $G$, preserving a probability measure $\mu$ on the compact metric space $X$, was also studied by Pavel Zorin-Kranich in his 2014 paper \cite{Pavel}.\\ 
In \cite{Bahring2019},  M. Bahring  proved, using a finitary version of the van der Corput lemma, the uniform convergence of the weighted average 
\begin{align*}
    \A_n(f)(x):=\frac{1}{m(F_n)}\int_{F_n}\xi(g)f(gx)dm(g)
\end{align*}
on a full $\mu$-measure subset of $X$
for continuous and ergodic action of a locally compact abelian (and thereby amenable) group $G$ on a compact space $X$, where $f\in L^2(X,\mu)$ and $\mu$ is probability measure on $X$ preserved by the action. A proof using the joining of two actions of $G$ on the measure space was given by \cite{ElAbdalaoui2021}. \\

One of the standard techniques for proving Wiener–Wintner theorems is to employ an appropriate version of JdLG decomposition.
 It states that given a unitary representation $\pi$ of a locally compact group $G$ on a Hilbert space $\CH$ equipped with an inner product $\langle \cdot,\cdot\rangle$, there exists a decomposition $\CH=\CH_{c}\oplus \CH_{wm}$, where $\CH_{c}=\{\xi \in \CH:\{U_{g}{\xi}:g \in G\} \text{ is totally bounded in } \CH\}$, is called the compact part and $\CH_{wm}=\CH^{\perp}_{c}$, is called the weak mixing part. If $G$ is amenable and equipped with a Haar measure $m$, we may further describe $\CH_{wm}$ as 
\begin{align*}
\CH_{wm}=\{\xi \in \CH:\lim_{n \to \infty}\frac{1}{m(F_{n})}\int_{F_{n}}|\langle \pi(g)\xi,\eta\rangle|dm(g)=0, ~\text{for all}~ \eta \in \CH\}.
\end{align*} 
In the present article, our set up for the non-commutative Wiener-Wintner theorem is the following: $G$ is a l.c.s.c. unimodular amenable group with a F{\o}lner sequence $(F_{n})_{n \in \mathbb{N}}$, acting by $*$-automorphisms  continuously and ergodically (meaning, $\alpha_{g}(x)=x, $  for all $ g \in G \text{ implies }x=k 1$ for some $k \in \mathbb{C}$) on a finite von Neuman algebra $\CM$ while preserving a faithful normal tracial state $\tau$ on $\CM$. The action of $G$ on $\CM$, given by $\alpha:G \rightarrow \Aut(\CM)$, may be extended unitarily to the non-commutative $L^2$-space associated to $\CM$ to give rise to a unitary group representation on the GNS Hilbert space $L^{2}(\CM,\tau)$ associated to $\CM$ with respect to $ \tau$. In this case, the corresponding Wiener-Wintner average is of the form
\begin{align*}
    \mathbb{A}_n(x):= \frac{1}{m(F_n)}\int_{F_n}\alpha_g(x)\phi(\pi(g) u)dm(g),
\end{align*}
where $(F_n)$ is an appropriate F\o lner sequence, $\pi:G \to \mathcal{U}_d$ is a finite dimensional unitary representation and $\phi$ is a complex valued continuous function on $\mathcal{U}_d$, the unitary group of operators on $\C^{d}$, and $u \in \CU_{d}$ is a unitary.\\
The aim of this article is to establish the convergence of the abovementioned noncommutative Wiener-Wintner averages in the b.a.u. sense:
\begin{thm}\label{main}
    Let G be a l.c.s.c. amenable unimodular group acting continuously on a finite von Neumann algebra $\CM$ by $*$-automorphisms preserving the trace.  Let $x\in L^p(\CM,\tau)$ and $\pi: G \to \mathcal{U}_d$ be a finite dimensional unitary representation of $G$ on $\mathbb{C}^d$ where $\mathcal{U}_d$ is the group of unitary operators on $\mathbb{C}^d$. Then for any continuous function $\phi$ on $\mathcal{U}_d$ the following weighted average 
    \begin{align}\label{Avg_2}
        \mathbb{A}_n(x):= \frac{1}{m(F_n)}\int_{F_n}\alpha_g(x)\phi(\pi(g) u)dm(g)
    \end{align}
    converges in b.a.u where $u \in \mathcal{U}_d$. Moreover if $x$ is in the weak-mixing part of the representation $\alpha$ then the average converges to $0$ in b.a.u. 
\end{thm}

In order to prove this result, we use the JdLG decomposition of the GNS space $L^{2}(\CM,\tau)$. We  establish the following  version of the JdLG decomposition for the finite von Neumann algebra.
\begin{thm}\label{JdLG}
Let $ \CM$ be a finite von Neumann algebra with a f.n trace $\tau$ and $ (\CM, G, \alpha) $ be a $\tau$-preserving noncommutative dynamical system.  Consider the following 
\begin{align*}
    \CM_{c} &= \{a \in \CM:a \widehat{1}\in L^{2}_{c}(\CM, \tau)\}, \text{ is a von Neumann subalgebra, and,}\\
\CM_{wm}&=\biggl\{a \in \CM:\lim_{n \rightarrow \infty}\frac{1}{m(F_{n})}\int_{F_{n}}\abs{\tau(b\alpha_{g}(a))}dm(g)=0, ~\text{ for all  } b\in \CM \biggr\}.
\end{align*}
Then $\CM$ can be written as 
$\CM=\CM_{c}\oplus \CM_{wm}$ such that 
\[ \overline{ \CM \cap L^{2}_{c}(\CM,\tau)} = L_{c}^{2}(\CM,\tau)    ~\text{ and }~  \overline{ \CM \cap L^{2}_{\text{wm}}(\CM,\tau) } ~=  L^{2}_{\text{wm}}(\CM,\tau) .\]
\end{thm}
Another result of significance in ergodic theory, classical and noncommutative, is the van der Corput inequality. 
It is often stated in literature in various forms, but of particular importance to us, is a van der Corput inequality for the setting of von Neumann algebras. A noncommutative version of this inequality was first proved by \cite{NSZJOT2003} for bounded functions $a:\mathbb{Z} \rightarrow \CA$, where $\CA$ is a C*-algebra, and was used in \cite{Lit2014}. A finitary version of this lemma was proved for bounded functions $f:G \rightarrow \C$, for a abelian group $G$, by \cite{Bahring2019}. In section 5 of this article, we prove the following van der Corput inequality:
\begin{lem}[A van der Corput inequality for amenable groups]\label{Lemma_vand}
Let $G$ be an l.c.s.c amenable group equipped with a right Haar measure $m$, and let $(F_n)$ be a right F\o lner sequence in $G$ and $\CM$ be a  von Neumann algebra. Let $f : G \to \CM$ be a bounded measurable function. Then for every $n, k \in \mathbb{N}$, the following estimate holds:
\begin{align*}\label{Vand_inequality}
\norm{
\frac{1}{m(F_n)} \int_{F_n} f(g)\, dm(g)
}^2 &\leq 
\frac{1}{m(F_k)^2} \int_{F_k} \int_{F_k}
\norm{ \left( \frac{1}{m(F_n)} \int_{F_n}
f( g)^* ~ f( gh_{1}^{-1}h_2) dm(g) \right) }
dm(h_1)\, dm(h_2) \\
&~~~+ 3\|f\|_\infty^2 \sup_{h \in F_k}
\frac{m(F_nh \triangle F_n)}{m(F_n)}  +   \|f\|^2_\infty \left( \sup_{h \in F_k}
\frac{m(F_nh \triangle F_n)}{m(F_n)} \right)^2.
\end{align*}
\end{lem}
The main novelty lies in our method. To the best of our knowledge, the JdLG decomposition in the form required here has not previously been available for amenable group actions on finite von Neumann algebras. We develop such a decomposition in the appropriate noncommutative setting. Then we prove a version of the noncommutative van der Corput lemma and a characterization of weak mixing along Reiter sequences, and use them to establish Theorem \ref{main}. The use of weak mixing along a Reiter sequence is, to the best of our knowledge, new even in the classical setting in this context of the Wiener-Wintner theorem. This approach overcomes the main obstruction in the noncommutative setting and yields the Wiener-Wintner theorem in full generality. 

Moreover, we point out that our results strengthen the main result of M.~B\"ahring~\cite{Bahring2019} on the uniform Wiener-Wintner theorem for unimodular locally compact amenable groups. We also provide an affirmative answer to a question raised by e. H. el Abdalaoui in \cite[Question 2.12]{ElAbdalaoui2021}.

\subsection{Outline of the paper}We first introduce some preliminaries on noncommutative $L^{p}$ spaces and dynamical systems in section \ref{preli}, followed by a discussion of some results of amenable groups and a proof of the JdLG decomposition for finite von Neumann algebras in section \ref{amena}  and section \ref{JdLG} respectively. In the last section, we first use the result of Cadhilac and Wang \cite{cadilhac2022noncommutative} to prove the a.u. and $\|.\|_{2}$ convergence of the unmodified average, which is followed by a Banach principle for the Wiener-Wintner averages. Then we provide a proof of the van der Corput lemma stated above and use these results to prove the convergence of the Wiener-Wintner averages in the b.a.u. sense.

We conclude the last section with a remark (see Remark \ref{class-comp}) comparing our results with the corresponding classical results and the techniques employed in the proofs and point out that our results improve some classical theorems regarding Wiener-Wintner Theorems.

\section{\textbf{Preliminaries}}\label{preli}
\noindent Throughout this article, let \(\mathcal{M}\) denotes a von Neumann algebra acting on a separable Hilbert space \(\mathcal{H}\),  i.e. $\mathcal{M}\subseteq \mathcal{B}(\mathcal{H})$. The operator norm inherited from \(\mathcal{B}(\mathcal{H})\) is written as \(\lVert \cdot \rVert_\infty\); when no confusion can arise, we abbreviate it simply by \(\lVert \cdot \rVert\). The commutant of \(\mathcal{M}\) is denoted by \(\mathcal{M}'\). We further write \(\mathcal{P}(\mathcal{M})\) for the collection of projections contained in \(\mathcal{M}\).

Let $\CM$ be a von Neumann algebra, and let $\CM_*$ denote its predual. Elements of $\CM_*$ are called \emph{normal linear functionals} on $\CM$. A functional $\phi \in \CM_*$ is said to be \emph{positive} if
\[
\phi(x)\geq 0, ~~\text{ for all }  x\in \CM_+,
\]
where $\CM_+$ denotes the cone of positive elements of $\CM$.
A positive functional $\phi \in \CM_*$ is called a \emph{state} if $ \phi(1)=1.$
It is called \emph{faithful} if for every $x\in \CM_+$,
\[
\phi(x)=0 \quad \Longleftrightarrow \quad x=0.
\]
For notational convenience, the term \emph{faithful normal} will be abbreviated by \emph{f.n.} Given an f.n.\ state $\phi$ on $\CM$, we denote by $L^2(\CM,\phi)$ the Hilbert space arising from the GNS construction associated with $\phi$. For each $a\in \CM$, we write $\widehat{a}$ for the corresponding vector in $L^2(\CM,\phi)$. Moreover, we identify $\CM$ with its image in $\CB(L^2(\CM,\phi))$ under the canonical GNS representation.

\subsection{\textbf{Non-commutative $L^p$-spaces  ($1\leq p\leq\infty $)}}
Let \(\mathcal{M} \subseteq \mathcal{B}(\mathcal{H})\) be a von Neumann algebra equipped with a faithful, normal, semifinite (f.n.s.) trace \(\tau\). An (possibly unbounded) operator 
\[
x: \mathcal{D}(x) \subseteq \mathcal{H} \to \mathcal{H},
\]
which is densely defined, closed, and affiliated with \(\mathcal{M}\), i.e,  satisfying \(u'^{*} x u' = x\) for every unitary \(u' \in \mathcal{M}'\), will be denoted as \(x \,\eta\, \mathcal{M}\). Such an operator \(x\) is said to be \(\tau\)-measurable if, for every \(\varepsilon > 0\), there exists a projection \(e \in \mathcal{P}(\mathcal{M})\) such that \(\tau(1-e) < \varepsilon\) and \(e\mathcal{H} \subseteq \mathcal{D}(x)\). The collection of all \(\tau\)-measurable operators affiliated with \(\mathcal{M}\) is then defined by
\[
L^{0}(\mathcal{M},\tau) := \{\, x \,\eta\, \mathcal{M} : x \text{ is } \tau\text{-measurable}\,\}.
\]

The space \(L^{0}(\mathcal{M},\tau)\) forms a \(^*\)-algebra under the adjoint operation, strong sum \(x+y := \overline{x+y}\), and strong product \(x \cdot y := \overline{xy}\), where \(\overline{x}\) denotes the closure of \(x\). Detailed descriptions may be found in \cite{stratila2019lectures, hiai2021lectures}.

Several natural topologies arise on \(L^{0}(\mathcal{M},\tau)\). The first relevant one is the \emph{measure topology}, whose neighborhoods are of the form:  $ x + \mathcal{N}(\varepsilon,\delta),
~~ \varepsilon,\delta > 0,\; x \in L^{0}(\mathcal{M},\tau),$
where
\[
\mathcal{N}(\varepsilon,\delta)
 := \{\, x \in L^{0}(\mathcal{M},\tau) :\text{there exist}~ e \in \mathcal{P}(\mathcal{M}) 
 \text{ with } \tau(1-e) < \delta \text{ and } \|exe\| < \varepsilon \,\}.
\]
A net \(\{x_i\}_{i\in I}\) converges to \(x\) in the measure topology precisely when, for every \(\varepsilon,\delta>0\), there exists \(i_0\in I\) such that for all \(i\ge i_0\) one may find \(e_i\in \mathcal{P}(\mathcal{M})\) with \(\tau(1-e_i) < \delta\) and
\[
\|\, e_i (x_i - x)e_i \,\| < \varepsilon.
\]

\noindent By \cite[Theorem 1]{nelson1974notes}, the space \(L^{0}(\mathcal{M},\tau)\) is complete in the measure topology. Furthermore, this topology is metrizable and  \(\mathcal{M}\) is dense in \(L^{0}(\mathcal{M},\tau)\).
The topologies  of interest are  the \emph{bilateral almost uniform} (b.a.u.)  and  almost uniform (a.u.) topology.

\begin{defn}\label{bau conv defn}
A net \(\{x_i\}_{i\in I}\subseteq L^{0}(\mathcal{M},\tau)\) is said to converge \emph{bilaterally almost uniformly} (b.a.u.) (resp. \emph{almost uniformly} (a.u.)) to \(x\in L^{0}(\mathcal{M},\tau)\) if, for every \(\varepsilon>0\), there exists a projection \(e \in \mathcal{P}(\mathcal{M})\) such that \(\tau(1-e) < \varepsilon\) and 
\begin{align*}
\lim_{i} \|\, e(x_i - x)e \,\| = 0 ~~~ (\text{resp. }  \lim_{i} \|\, (x_i - x)e \,\| = 0).
\end{align*}
\end{defn}

 Clearly, almost uniform(a.u.) convergence of a net implies the bilateral almost uniform(b.a.u) convergence. 
The trace \(\tau\) on \(\mathcal{M}_+\) extends in the standard way to \(L^{0}(\mathcal{M},\tau)_+\) by
\[
\tau(x) := \int_0^{\infty} \lambda \, d\tau(e_\lambda),
\]
where \(x = \int_{0}^{\infty} \lambda\, de_\lambda\) is the spectral decomposition of \(x\).
 For $1\leq p <\infty$, 
consider the set 
$$L^p(\mathcal{M}, \tau):=\{x\in L^0(\mathcal{M},\tau): \tau(|x|^p)<\infty\}.$$
This forms a normed linear space with norm $\|x\|_p:=\tau(|x|^p)^{1/p}$. For $p=\infty$, we define $L^\infty(\mathcal{M}, \tau):= \mathcal{M}$ with norm $\|\cdot\|_\infty$, the usual operator norm $\|\cdot\|$ in $\mathcal{M}$. The space $L^p(\mathcal{M}, \tau)$ is called the \textit{non-commutative $L^p$- space}. 
We recall the following theorem from \cite{hiai2021lectures}.
\begin{thm}\label{Density_thm}
    For every \(p \in [1,\infty]\), $L^p(\mathcal{M}, \tau)$ is a Banach space with the norm
\(\|\cdot\|_p\). In particular, $L^2(\mathcal{M}, \tau)$ is a Hilbert space with the inner
product
\[
\langle x,y\rangle = \tau(y^*x).
\]
Moreover, $ \mathcal{M} \cap L^1(\mathcal{M}, \tau)$
is dense in $L^p(\mathcal{M}, \tau)$ for any \(p \in [1,\infty)\).
\end{thm}



\subsection{Noncommutative dynamical systems}
\begin{defn}\label{nc dyn sys} Let \((E,\|\cdot\|)\) be a real ordered Banach space.
Then a \emph{noncommutative dynamical system} is a triple \((E,G,\gamma)\), where \(\gamma : G \to \mathcal{B}(E)\) is a mapping satisfying \(\gamma_s \circ \gamma_t = \gamma_{st}\) for all \(s,t \in G\), and the following conditions hold:
\begin{enumerate}
    \item for every \(a \in E\), the orbit map \(s \mapsto \gamma_s(a)\) is continuous on \(G\), where \(E\) is equipped with the norm topology (and the \(w^*\)-topology in the case \(E=\mathcal{M}\));
    \item the operators are uniformly bounded, i.e., \(\sup_{s \in G}\|\gamma_s\| < \infty\);
    \item positivity is preserved: if \(a \ge 0\), then \(\gamma_s(a) \ge 0\) for all \(s\in G\).
    
\end{enumerate}
\end{defn}
\noindent We further assume that if $E = \CM$, $\gamma_{s}$ is a *-automorphism on $\CM$ for all $s \in G$.
Further, suppose $ \CM$ is a finite von Neumann algebra with a f.n state $ \tau.$ We say that $(\CM, G, \alpha)$ is $\tau$-preserving if $ \tau \circ \alpha_g = \tau$, for all $ g \in G.$

\noindent Let $(\CM, G, \alpha)$ be a noncommutative dynamical system, where $G$ is a locally compact second countable amenable group acting on a finite von Neumann algebra $\CM$ with a faithful normal tracial state $\tau.$ 
The action is given by a homomorphism
\[
\alpha : G \to \mathrm{Aut}(\CM),
\]
where each $\alpha_g$ is a $*$-automorphism of $\CM$. This means that for every $g \in G$ and every $x \in \CM$,
\[
\alpha_g(x^*) = \alpha_g(x)^*.
\]
We also assume that the action preserves the trace, that is, $ \tau \circ \alpha_g = \tau \quad \text{for all } g \in G$.  We also have 
$\displaystyle\sup_{g \in G} \|\alpha_g\| \leq 1. $  Now let $x \in \CM \cap L^1(\CM, \tau)$. Then
\[
|\alpha_g(x)|^2=\alpha_g(x)^*\alpha_g(x) = \alpha_g(x^*x)=\alpha_g(|x|^2)=\alpha_g(|x|)^2.
\]
Hence, $ |\alpha_g(x)|=\alpha_g(|x|)$  and $|\alpha_g(x)|^n=\alpha_g(|x|^n) $ for all $ n \in \mathbb{N}$.
Hence,  by the continuous functional calculus and Weierstrass' theorem, 
\[|\alpha_g(x)|^p=\alpha_g(|x|^p), ~~~ \text{ for all }~~ 1\leq p < \infty.\]\label{alphamodx}
Thus,\begin{align*}
    \|\alpha_g(x)\|^p_p = \tau(|\alpha_g(x)|^p)=\tau(\alpha_g(|x|^p))=\tau(|x|^p)=\|x\|^p_p
    \implies \|\alpha_g(x)\|_p = \|x\|_p.
\end{align*}
\noindent Since $\CM \cap L^1(\CM, \tau)$ is dense in $L^p(M, \tau)$, it follows that $\alpha_g$ extends uniquely to a bounded operator on $L^p(\CM, \tau).$

\subsection{Amenable Groups}\label{amena}
In this subsection we discuss some basics of amenable groups
\begin{defn}
Let $G$ be a locally compact second countable (l.c.s.c.) group equipped with a left (resp. right) Haar measure $m$.
The group $G$ is called \emph{amenable} if it admits a left (resp. right) F\o lner sequence in the sense of the following definition.
\begin{enumerate}
\item[(i)] 
A sequence $(F_n)$ of nonempty compact subsets of $G$ is called a \emph{left (resp. right) F\o lner sequence} if
\begin{align*}
\frac{m(KF_n \,\triangle\, F_n)}{m(F_n)} \longrightarrow 0~~
(\text{resp.} \frac{m(F_nK \,\triangle\, F_n)}{m(F_n)} \longrightarrow 0)
\qquad \text{as } n \to \infty
\end{align*}
for every compact set $K \subset G$.
\item[(ii)] Let $ C> 0$, a  F\o lner sequence $(F_n)$ is said to be $C$-\textit{tempered}  if for all $n \in \N$, $$m\bigl(\bigcup_{k<n}F^{-1}_kF_n\bigr) \leq Cm(F_n).$$

\item[(iii)] A F\o lner sequence $(F_n)$ is said to satisfy the \textit{Tempelman} condition if there exists $C>0$ such that for all $n\in \N$, $$m(\bigcup_{k\leq n}F_k^{-1}F_n) \leq Cm(F_n).$$
\item[(iv)]  We also require the definition of admissible F\o lner sequences, for this definition we refer to \cite{cadilhac2022noncommutative}.
\end{enumerate}
\end{defn}

Possibly, some of the following results concerning amenable groups are folklore in the literature. However, since we have been unable to locate a reference containing these statements in the required form, we include the proofs for the reader’s convenience.
\begin{lem}\label{F'_n_construction_1}
 Let $G$ be a l.c.s.c amenable unimodular group with F\o lner sequence $(F_n)$. Given a compact symmetric set $K \subseteq G, e\in K$ and $k\in \N, k\geq 2$  there exists $ E_n \subseteq F_n$ and $N(K,k) \in \N$ such that  for all $ n \geq N(K,k)$, $E_nK^k \subseteq F_n$ and $\frac{m(F_n\setminus E_n)}{m(F_n)} \leq \frac{1}{k}.$   
\end{lem}
\begin{proof}
    Let us define $E_n:= \{x\in F_n: xK^k \subseteq F_n\}=F_n \cap F_nK^{k}.$ Clearly, $E_n \subseteq F_n$ and $E_nK^k \subseteq F_n$, so we use the F\o lner condition to obtain the inequality.
\end{proof}
\begin{lem}\label{F'_n_construction_2}
Let $G$ be a l.c.s.c amenable unimodular group with F\o lner sequence $(F_n)$ which satisfies Tempelman condition   (i.e.  there exists $ C>0$ such that $m(\displaystyle\bigcup_{k\leq n} F_k^{-1}F_n) \leq Cm(F_n)~~\text{ for all } n \in \N$). Given a symmetric compact set $K \subseteq G$ with $e\in K$ and $k\geq 2$, there exists $ N(K,k) \in \N$ and a compact set $F'_{n,k} \subseteq F_n$ such that for all $ n \geq N(K,k)$ it has the following properties:\\
\begin{enumerate}
    \item $\frac{m(F_n \setminus F'_{n,k})}{m(F_n)} \leq \frac{1}{k}$
    \item $m({F'^{-1}_{n,k}}F'_{n,k}) \leq C\frac{k}{k-1}m(F'_{n,k}) \leq 2Cm(F'_{n,k})$
    \item $\frac{m(K{F'^{-1}_{n,k}}F'_{n,k} \Delta {F'^{-1}_{n,k}}F'_{n,k})}{m({F'^{-1}_{n,k}}F'_{n,k})} \leq \frac{2C}{k}$.
\end{enumerate}
\end{lem}
\begin{proof}
 Lemma \ref{F'_n_construction_1} implies that  there exists $ N(K,k) \in \N$ such that 
 $$E_nK^k \subseteq F_n ~~\text{and}~~ \frac{m(F_n\setminus E_n)}{m(F_n)} \leq \frac{1}{k},  ~~\text{ for all } n\geq N(K,k).$$ 
 Define $A_i:= K^i{E_n}^{-1}E_nK^i~~\text{for}~~ i\in \{1,2,\cdots,k\}$. Since $K$ is symmetric, it follows that   $A_i \subseteq F_n^{-1}F_n$. 
It also follows that  $A_i \subseteq A_{i+1}$  for $  i\in \{1,2,\cdots, k-1\}$  and 
 \begin{align}\label{Inequality_of_Tempelman_condition}
     ~~ \sum_{i=0}^{k-1}m(A_{i+1}\setminus A_i)=m(A_k \setminus A_0) \leq m(F_n^{-1}F_n) \leq Cm(F_n).
 \end{align}
Hence,  there exists $ i \in \{1,2,\cdots,k-1\} $ such that 
 $$ 
 m(A_{i+1}\setminus A_i)\leq \frac{Cm(F_n)}{k}.
 $$ 
 For this $i$ we define $F'_{n,k}:= E_nK^i$. Clearly, $E_n \subseteq F'_{n,k} \subseteq F_n$, which implies that 
 \begin{align}\label{Property_1}
     F_n\setminus F'_{n,k} \subseteq F_n\setminus E_n \implies \frac{m(F_n \setminus F'_{n,k})}{m(F_n)} \leq \frac{1}{k}.
 \end{align}
 Since $F'_{n,k} \subseteq F_n$, ${F'^{-1}_{n,k}}F'_{n,k} \subseteq {F^{-1}_n}F_n$, $m({F'^{-1}_{n,k}}F'_{n,k}) \leq m({F^{-1}_n}F_n) \leq Cm(F_n).$
 Therefore, $(2)$ follows from Ineq. (\ref{Property_1}).\\
 Note that ${F'^{-1}_{n,k}}F'_{n,k} = A_i$ and ${KF'^{-1}_{n,k}}F'_{n,k} \subseteq A_{i+1}.$
 Hence,
 \begin{align*}
\frac{m(K{F'^{-1}_{n,k}}F'_{n,k} \Delta {F'^{-1}_{n,k}}F'_{n,k})}{m({F'^{-1}_{n,k}}F'_{n,k})} &\leq \frac{m(A_{i+1}\setminus A_i)}{m({F'^{-1}_{n,k}}F'_{n,k})}\\ 
&\leq \frac{Cm(F_n)}{k}\cdot \frac{1}{m({F'^{-1}_{n,k}}F'_{n,k})}\\
&=\frac{C}{k}\frac{m(F_n)}{m(F'_{n,k})}\cdot \frac{m(F'_{n,k})}{m({F'^{-1}_{n,k}}F'_{n,k})}
\leq \frac{2C}{k}.
 \end{align*}
 The last inequality from the second last line follows the fact that $m(F'_{n,k})=m(y^{-1}F_{n,k})\leq m({F'^{-1}_{n,k}}F_{n,k})$ and the Ineq. (\ref{Property_1}).
\end{proof}
Suppose $G$ is a l.c.s.c group. We can write $G$ to be the union of countably many compact sets. Further, we can take those compact sets to be symmetric, increasing and containing $e$. Let $G=\displaystyle\bigcup_{r\geq1} K_r$. Now given admissible F\o lner sequence $(F_n)$ ,using these $K_r$ we will construct another F\o lner sequence $(F'_n)$ such that ${F'^{-1}_n}F'_n$ is again a F\o lner sequence and $(F'_n)$ satisfies Tempelman condition. 
\begin{thm}\label{existence_of_F'_n}
    Let $G$ be a l.c.s.c amenable unimodular group with F\o lner sequence $(F_n)$ that satisfies Tempelman condition. Then there exists $F'_n \subseteq F_n$ such that for large enough $n$ we have the following: 
    \begin{enumerate}
    \item $\displaystyle\lim_{n\to \infty}\frac{m(F_n\setminus F'_n)}{m(F_n)}=0.$
        \item $({F'^{-1}_n}F'_n)$ is a F\o lner sequence. 
        \item $(F'_n)$ is F\o lner sequence that satisfies the Tempelman condition for large $n \in \mathbb{N}$.
    \end{enumerate}
\end{thm}
\begin{proof}
 Applying Lemma \ref{F'_n_construction_2} for each $r\in \N $ with $r\geq 2$, we obtain that there exist $N_r \in \mathbb{N}, ~~N_r < N_{r+1}$ and $F'_{n,r} \subseteq F_n$ such that for all $n \geq N_r$, we have the following
 \begin{align}\label{three_inequalities}
     \frac{m(F_n \setminus F'_{n,r})}{m(F_n)} \leq \frac{1}{r},~~ m({F'^{-1}_{n,r}}F'_{n,r})  \leq 2Cm(F'_{n,r}) ~~\text{and}~~ \frac{m(K_r{F'^{-1}_{n,r}}F'_{n,r} \Delta {F'^{-1}_{n,r}}F'_{n,r})}{m({F'^{-1}_{n,r}}F'_{n,r})} \leq \frac{2C}{r}.
 \end{align}
For each $ n \in \N $  with $ n\geq N_2$, define $r(n):=max\{t\in \mathbb{N}:N_t \leq n\}.$ Clearly, $N_{r(n)} \leq n$. Given $l \in \mathbb{N}$, we note that   $r(n) \geq l$,  for all $ n\geq N_l$. Hence, $r(n) \to \infty$ as $n\to \infty.$\\
Let, $F'_n= F'_{n,r(n)}$. Therefore, $(1)$ immediately follows from Inequality in (\ref{three_inequalities}).

\medskip
\noindent
Suppose  $h\in G$. Since $G=\displaystyle\bigcup_{r\geq 1} K_r$'s with  $K_r$'s are increasing and $r(n) \to \infty$ as $n \to \infty$, there exists $M \in \N$ such that $h \in K_{r(n)}$, for all $n \geq M$. Therefore,  for all $n\geq M$, we observe the following: 
\begin{align*}
  &h{F'^{-1}_n}F'_n\setminus {F'^{-1}_n}F'_n \subseteq K_{r(n)}{F'^{-1}_{n,r(n)}}F'_{n,r(n)} \setminus {F'^{-1}_{n,r(n)}}F'_{n,r(n)}\\
  \implies & \frac{m(h{F'^{-1}_n}F'_n\setminus {F'^{-1}_n}F'_n)}{m({F'^{-1}_n}F'_n)} \leq \frac{2C}{r(n)}, ~~ \text{by Inequality in } (\ref{three_inequalities}).   
\end{align*}
Further, we note that 
\begin{align*}
 &{F'_n}^{-1}F'_n\setminus h{F'_n}^{-1}F'_n = h(h^{-1}{F'_n}^{-1}F'_n\setminus {F'_n}^{-1}F'_n)\\
 \implies & m({F'_n}^{-1}F'_n\setminus h{F'_n}^{-1}F'_n)=m(h^{-1}{F'_n}^{-1}F'_n\setminus {F'_n}^{-1}F'_n)   
\end{align*}
Since $K_r$'s are symmetric, $h^{-1} \in K_{r(n)}$. Hence, similarly, we obtain the following  
\begin{align*}
   \frac{ m({F'_n}^{-1}F'_n\setminus h{F'_n}^{-1}F'_n) }{m({F'_n}^{-1}F'_n)} \leq \frac{2C}{r(n)}.
\end{align*}
Consequently, we conclude that
\begin{align*}
    \frac{m(h{F'_n}^{-1}F'_n \Delta {F'_n}^{-1}F'_n)}{m({F'_n}^{-1}F'_n)} \leq \frac{4C}{r(n)},
\end{align*}
which proves $(2).$\\
Let $g\in G$. Using the triangle inequality of the symmetric difference $\Delta$ we have 
\begin{align}\label{triangle_ineq_of_Delta}
    m(gF'_n \Delta F'_n) \leq m(gF'_n \Delta gF_n) + m(gF_n \Delta F_n) + m(F_n \Delta F'_n).
\end{align}
Now, first inequality of (\ref{three_inequalities}) gives us $m(F'_n) \leq m(F_n) \leq 2m(F'_n)$.
Therefore, from the Ineq.(\ref{triangle_ineq_of_Delta}), we get
\begin{align*}
 \frac{m(gF'_n \Delta F'_n)}{m(F'_n)} &\leq \frac{m(gF'_n \Delta gF_n)}{m(F'_n)} + \frac{m(gF_n \Delta F_n)}{m(F'_n)}+ \frac{m(F_n \Delta F'_n)}{m(F'_n)}\\
 &\leq 2\frac{m(F_n \setminus F'_n)}{m(F_n)}+2\frac{m(gF_n \Delta F_n)}{m(F_n)} + 2\frac{m(F_n \setminus F'_n)}{m(F_n)}\\
 &\leq 4\frac{m(F_n \setminus F'_n)}{m(F_n)} + 2\frac{m(gF_n \Delta F_n)}{m(F_n)}.
\end{align*}
By the F\o lner property of $(F_n)$ and using $(1)$ of this theorem,   it follows that 
$(F'_n)$ is F\o lner sequence. Now it is immediate that $(F'_n)$ satisfies the Tempelman condition from the fact that $F'_n \subset F_n$ and $m(F_n) \leq 2m(F'_n)$ for large enough $n \in \N.$ This completes our proof.  
\end{proof}



\section{\textbf{Jacob de Leeuw Glicksberg decomposition  for amenable group actions on  finite von Neumann algebra}}\label{JdLG}
\noindent In this section we study a JdLG decomposition for finite von Neumann algebras for the action by an amenable group on a finite von Neumann algebra.
In October 2025, Sohail Farhangi,Y.Kuznetsova and Mickey Barthmann proved the following version of JdLG decomposition for the amenable group action on a Banach space.
\begin{thm}\label{JdlG_decomposition}\cite{barthmann-farhangi-kuznetsova2025jacobsdeleeuwglicksbergdecomposition}
Let $G$ be a locally compact amenable group, $m$ a left Haar measure,
$ (F_i)_{i \in I}$ a left F\o lner net, $E$ a Banach space, and
$\pi$ a relatively weakly compact representation of $G$ on $E$.
Then $\pi$ is JdLG-admissible, and it follows that 
\begin{align*}
E_c :&= \{\, \xi \in E : \pi(G)\xi \text{ is relatively compact in } E \,\}  \text{ and }\\
E_{\mathrm{wm}} :&= \left\{\, \xi \in E :
\lim_{i} \frac{1}{m(F_i)}
\int_{F_i} \bigl| \langle \pi(g)\xi, x' \rangle \bigr| \, dm(g) = 0
\text{ for all } x' \in E' \right\}.
\end{align*}
\end{thm}
$E_c$ is referred to as the \emph{ compact } part and  $ E_{\mathrm{wm}}$ is refered to as the \emph{weak mixing} part.
 When  $E=\CH$, a Hilbert space, and $\pi$ to be a unitary representation of $G$ on $\CH$. Since $\| \pi(g) \|=1$ for all $g \in G$, and the closed unit ball of $\CB(\CH)$ under WOT is compact, and the representation is JdLG admissible by  \cite[Example 16.25, Chapter 16]{Eisner2015}. Moreover, we have $ \CH= \CH_c \oplus \CH_{\text{wm}}$.\\
 We recall the following definition:
 \begin{defn}
     Let $(X,d)$ be a metric space. A set $\CB \subset X$ is totally bounded if, for all $\epsilon >0$, there exists a finite set $\CF(\epsilon,\CB)=\{x_{1},x_2...,x_n\} \subset X$ such that for all $b \in \CB$, there exists $x_i \in \CF(\epsilon,\CB)$ such that $d(x_{i},b)<\epsilon.$
 \end{defn}
 Therefore,  one may characterise the compact part in terms of totally boundedness. We will also use the following characterisation of compactness:
 \begin{prop}\label{cptchar2}
Let $(X,d)$ be a metric space. Let $A \subset X$. Then, $\overline{A} \subset X$ is compact iff every sequence $(x_{n})_{n\in\mathbb{N}}$ in $A$ admits a convergent subsequence $(x_{n_k})_{k\in \mathbb{N}}$.
 \end{prop}
 Let $\CM$ be a finite von Neumann algebra equipped with a faithful normal tracial state $\tau$ and  $G$ be a l.c.s.c amenable group with left Haar measure $m$ and a F\o lner sequence $(F_n)_{n \in \mathbb{N}}$. Let $(\CM, G, \alpha)$ be a $\tau$-preserving noncommutative dynamical system.

Let $L^2(\CM,\tau)$ be  the GNS Hilbert space associated with $(\CM,\tau)$. 
Suppose $\widehat{1}$ denotes the cyclic and separating vector associated with $\tau.$
Consider the family of unitaries  $ \{ u_g : g \in G \}$ in $ L^2(\CM, \tau)$ defined by 
$$ u_g \widehat{ x } = \widehat{\alpha_g(x)}, ~\text{ for all } g \in G.$$
We note that $u_g$ can be thought as an extension of  the $*$-automorphism $\alpha_g$ of $\CM$ to $ L^2(\CM, \tau)$. Thus, with mild abuse of notation, we sometime use the same notation $ \alpha_g$ for the extension to $L^2(\CM, \tau)$.
We apply the JdLG decomposition  to the family $ \{ u_g: ~g \in G\} \subset \CU(L^{2}(\CM, \tau))$ and write $ L^2_c(\CM, \tau)$  and $L^2_{ \text{wm} }(\CM, \tau)$ for the compact and weak mixing part respectively. Now we like to obtain a JdLG decomposition for 
 $ (\CM,  G, \alpha)$. Consider the following sets;
\begin{align*}
    \CM_{c} &= \{a \in \CM:a \widehat{1}\in L^{2}_{c}(\CM, \tau)\}, \text{ is a von Neumann subalgebra, and,}\\
\CM_{wm}&=\biggl\{a \in \CM:\lim_{n \rightarrow \infty}\frac{1}{m(F_{n})}\int_{F_{n}}\abs{\tau(b\alpha_{g}(a))}dm(g)=0,~\text{for all}~ b\in \CM \biggr\}.
\end{align*}
We wish to prove the following theorem, which is a version of the JdLG decomposition for finite von Neumann algebras. 
\begin{thm}\label{JdLG}
Let $ \CM$ be a finite von Neumann algebra with a f.n trace $\tau$ and $ (\CM, G, \alpha) $ be a $\tau$-preserving noncommutative dynamical system. Then $\CM$ can be written as 
$\CM=\CM_{c}\oplus \CM_{wm}$ such that 
\[ \overline{ \CM \cap L^{2}_{c}(\CM,\tau)} = L_{c}^{2}(\CM,\tau)    ~\text{ and }~  \overline{ \CM \cap L^{2}_{\text{wm}}(\CM,\tau) } ~=  L^{2}_{\text{wm}}(\CM,\tau) .\]
\end{thm}
\begin{rem}
Since the definition of the compact part $\CM_{c}$ is not dependent on the F{\o}lner sequence, the definition of the weak mixing part $\CM_{wm}$ is also not dependent on the F{\o}lner sequence, by the uniqueness of the orthogonal complement of the closed Hilbert subspace $L_{c}^{2}(\CM, \tau)$.
\end{rem}
We prove the above Theorem \ref{JdLG} through the sequence of results.  We start with the following result.
\begin{lem}\label{xstarcpt} 
Let $\CM$ be finite von Neumann algebra with a f.n trace $\tau$ and $ (\CM, G, \alpha)$ be $\tau$-preserving noncommutative dynamical system. Suppose  
 $x \in L^{2}_{c}(\CM, \tau)$, then $x^{*}\in L^{2}_{c}(\CM, \tau)$.
\end{lem}
\begin{proof}
Let $(g_{n})_{n\in \mathbb{N}}$ be a sequence, then there exists a subsequence $(n_k)_{k\in \mathbb{N}}$ and $a\in L^{2}(\mathcal{M},\tau)$ such that 
\begin{align*}
    \lim_{k \rightarrow \infty}\|\alpha_{g_{n_k}}(x)-a\|_{2}=0.
\end{align*}
By a property of noncommutative $L^{2}$-spaces, we know that $\|b^{*}\|_{2}=\|b\|_{2}$ for all $b \in L^{2}(\mathcal{M},\tau)$.\\
We note that $\alpha_{g}(x^{*})=(\alpha_{g}(x))^{*}$ for all $g \in G \text{ and } x\in L^{2}(\CM,\tau)$. Indeed, given $x \in L^{2}(\mathcal{M},\tau)$ and $\epsilon > 0$, there exists $b \in \CM$ such that $\|x-b\|_{2}<\epsilon$, and since $\alpha_{g}$ is a contraction for all $g \in G$, we have that $\|\alpha_g(x)-\alpha_{g}(b)\|_{2}<\epsilon$ for all $g 
\in G$. Since $b \in \CM,\alpha_{g}(b^{*})=(\alpha_{g}(b))^{*}$ for all $g \in G$, and 
\begin{align*}
    \|\alpha_{g}(x^{*})-(\alpha_{g}(x))^{*}\|_{2}&=\|\alpha_{g}(x^{*})-(\alpha_{g}(x))^{*}+\alpha_{g}(b^{*})-(\alpha_{g}(b))^{*}\|_{2}\\
    &\leq\|\alpha_{g}(x^{*})-\alpha_{g}(b^{*})\|_{2}+\|(\alpha_{g}(x))^{*}-(\alpha_{g}(b))^{*}\|_{2}\\
    &\leq \frac{\epsilon}{2}+\frac{\epsilon}{2}=\epsilon.
    \end{align*}
 Therefore, we obtain the following:
\begin{align*}
    \lim_{k \rightarrow \infty}\|\alpha_{g_{n_k}}(x^{*})-a^{*}\|_{2}=0.
\end{align*}
Therefore, $x^{*}$ is also in $L^{2}_{c}(\CM, \tau)$. 
 
\end{proof}

\noindent We note that the above decomposition is merely a vector space direct sum. However, the compact part of the decomposition has more structure than that of a vector space.
\begin{lem}
$\CM_{c}$ is a von Neumann subalgebra of $\CM$.
\end{lem}
\begin{proof}
Let $a,b \in \CM_{c}$.
Given the sequence $(g_{n})_{n\in \mathbb{N}}$, and $a \in \CM_c$, there exists a  subsequence $(g_{n_{r}})_{r \in \mathbb{N}}$ such that $u_{g_{n_r}}(\widehat{a})$ is Cauchy. Given the subsequence $(g_{n_{r}})_{r \in \mathbb{N}}$, there exists a further subsequence $(g_{n_{r_{l}}})_{l \in \mathbb{N}}$ such that $u_{g_{n_{r_{l}}}}(\widehat{b})$ is Cauchy. Since every subsequence of a Cauchy subsequence is also Cauchy, we have that $u_{g_{n_{r_{l}}}}(\widehat{a})$ is also Cauchy. We call this subsequence $(g_{n_{k}})_{k \in \mathbb{N}}$.

\noindent Choosing this new subsequence, we obtain that $u_{g_{n_{k_{l}}}}(\widehat{a})+u_{g_{n_{k_{l}}}}(\widehat{b})=u_{g_{n_{k_{l}}}}(\widehat{a+b})$ is Cauchy in $L^{2}$ norm, and consequently, that $a+b \in \CM_{c}$. \\
In order to show that $ab \in \CM_{c}$, we need to show that $(\alpha_{g_{n_{k}}}(\widehat{ab}))_{k \in \mathbb{N}}$ is Cauchy. Let $k,k' \in \mathbb{N}$. Then,
\begin{align*}
 &\norm{\alpha_{g_{n_{k}}}(a)\alpha_{g_{n_{k}}}(b)-\alpha_{g_{n_{k'}}}(a)\alpha_{g_{n_{k'}}}(b)}_2\\=&\norm{(\alpha_{g_{n_{k}}}(a)-\alpha_{g_{n_{k'}}}(a))\alpha_{g_{n_{k}}}(b) +(\alpha_{g_{n_{k'}}}(a))(\alpha_{g_{n_{k}}}(b)-\alpha_{g_{n_{k'}}}(b))}_2  \\
 \leq&\norm{(\alpha_{g_{n_{k}}}(a)-\alpha_{g_{n_{k'}}}(a))\alpha_{g_{n_{k}}}(b)}_{2} +\norm{(\alpha_{g_{n_{k'}}}(a))(\alpha_{g_{n_{k}}}(b)-\alpha_{g_{n_{k'}}}(b))}_2\\
  \leq &\max\{{\|a\|,\|b\|\}} \bigl(\norm{(\alpha_{g_{n_{k}}}(a)-\alpha_{g_{n_{k'}}}(a))}_2 + \norm{(\alpha_{g_{n_{k}}}(b)-\alpha_{g_{n_{k'}}}(b))}_{2}\bigr), ~~\text{(Since $a,b \in \CM)$}.
\end{align*}
As $k,k' \rightarrow \infty$, we get the desired result.\\
By Lemma \ref{xstarcpt}, $\CM_c$ is closed under taking adjoints.
Now we show that it is closed in the SOT. Let $a \in \overline{\{m \in \CM:m\widehat{1}\in L_{c}^{2}(\mathcal{M},\tau)\}}^{\text{SOT}},$ then for  given $\epsilon > 0$ and $\widehat{1} \in L^{2}(\CM, \tau)$, there exists $b_\epsilon \in \{a \in \CM:a\widehat{1} \in L_{c}^{2}(\CM,\tau)\}$ such that
\begin{align*}
\|(a-b_\epsilon) \widehat{1}\|_{2}<\frac{\epsilon}{2}.
\end{align*}
Since $ \epsilon$ is arbitrary, hence, $ a\widehat{1} \in \overline{ L_c^2(\CM, \tau)}^{ \norm{\cdot}_2}.$ But  $ L_c^2(\CM, \tau)$ is a closed set, so, $  a\widehat{1}  \in L^2(\CM, \tau)$. Consequently it implies $ a \in \CM_c$.
\end{proof}

\noindent We recall the Powers-Str{\o}mer inequality from \cite{AnantharamanPopaII1} in order to prove the next lemma:
\begin{prop}[Powers--Størmer inequality]\label{PowerStormer} 
Let $\CM$ be a finite von Neumann algebra with a f.n trace $\tau$. 
Let $x, y \in L^{2}(\CM, \tau)_{+}$, then,
\begin{align*}
    \|x-y\|_2^{2} \leq \|x^{2}-y^{2}\|_1\leq\|x-y\|_2\|x+y\|_{2}.
\end{align*}
Consequently, it follows that for $x, y \in L^{2}(\CM,\tau)$
\begin{align*}
    \||x|-|y|\|_2^{2}\leq2max\{\|x\|_{2},\|y\|_{2}\}\|x-y\|_2.
\end{align*}
\end{prop}
If $b \in \CM\cap L^{1}(\CM,\tau)$, we know that $\alpha_{g}(|b|)=|\alpha_{g}(b)|$ for all $g \in G$. We note, by the following corollary, that the same is true when $\alpha_{g}$ is extended to $L^{2}(\CM,\tau)$ as a unitary for all $g \in G$. 
\begin{cor}\label{mod_preserving}
 $\alpha_{g}(|b|)=|\alpha_{g}(b)|$ for all $b \in L^{2}(\mathcal{M},\tau)$.
\end{cor}
\begin{proof}
First we claim  that the  map $x \mapsto |x|$, from $L^{2}(\mathcal{M},\tau)$ to itself, is continous in the $L^{2}$-norm.
Let $(x_n)_{n \in \mathbb{N}}$ be a sequence in $L^{2}(\CM,\tau)$ converging to $x \in L^{2}(\CM,\tau)$. We know that convergent sequences are bounded, so let $sup_{n\in \mathbb{N}}\|x_{n}\|_2=M$. Then, by Powers St{\o}rmer inequality, we have
\begin{align*}
    \||x_{n}|-|x|\|_2^{2}&\leq2max\{\|x_{n}\|_{2},\|x\|_{2}\}\|x_{n}-x\|_2\\
    &\leq2max\{M,\|x\|_{2}\}\|x_{n}-x\|_{2}.
\end{align*}
As $n \rightarrow\infty$, we prove the claim.\\
We know that $\alpha_{g}(|x|)=|\alpha_{g}(x)|$ holds for all $x \in \CM$. By the density of $\CM$ in $L^{2}(\mathcal{M},\tau)$, given $x \in L^{2}(\mathcal{M},\tau)$, we find a sequence $(y_{n})_{n \in \mathbb{N}} \subset \CM$ such that $y_{n}$ converges to $x$ in the $L^{2}$-norm. So, for all $g \in G$,
\begin{align*}
 &\|\alpha_g(|x|)-|\alpha_g(x)|\|_{2}\\&=\|\alpha_g(|x|)-\alpha_g(|y_{n}|)+\alpha_g(|y_n|)-|\alpha_g(x)|\|_2\\
&\leq \|\alpha_g(|x|)-\alpha_g(|y_{n}|)\|_2+\|\alpha_g(|y_n|)-|\alpha_g(x)|\|_2\\
&\leq \|\alpha_g(|x|)-\alpha_g(|y_{n}|)\|_2+\||\alpha_g(y_n)|-|\alpha_g(x)|\|_2, ~ (\text{as } y_n \in \CM \text{ and }   |\alpha_{g}(y_{n})|=\alpha_{g}(|y_{n}|)).
\end{align*}
By our claim, and the fact that for all $g \in G$, $\alpha_g$ is a contraction on $L^2(\mathcal{M},\tau)$, $y_{n} \rightarrow x$ as $n \rightarrow \infty$ in $L^{2}$-norm implies $|y_{n}| \rightarrow |x|$ and $|\alpha_g(y_{n})| \rightarrow |\alpha_g(x)|$ in $L^{2}$-norm as $n \rightarrow \infty$. Taking the limit $n \rightarrow \infty$ above, we prove the corollary.
\end{proof}
\begin{lem}\label{modcpt}
Suppose $x\in L_{c}^{2}(\mathcal{M},\tau)$, then it follows that $  |x| \in L_{c}^{2}(\mathcal{M},\tau)$. 
\end{lem}
\begin{proof}
First note that as  $x \in L_{c}^{2}(\CM,\tau)$, given a sequence $(g_{n})_{n\in \mathbb{N}}$, there exists a subsequence $(g_{n_{k}})_{k\in \mathbb{N}}$ such that $\{\alpha_{g_{n_{k}}}(x)\}_{k\in \mathbb{N}}$  is a Cauchy sequence with respect to the $L^{2}$-norm.\\
Now we apply the  Powers-St\o rmer's inequality (Proposition \ref{PowerStormer}) to the elements  $\alpha_{g_{n_{k}}}(|x|)$ and $\alpha_{g_{n_{l}}}(|x|)$, where $k,l\in \mathbb{N}$. Since $\alpha_{g}(|b|)=|\alpha_{g}(b)|$ for all $b \in L^{2}(\CM)$, we obtain,
\begin{align*}
&\|\alpha_{g_{n_{k}}}(|x|)-\alpha_{g_{n_{l}}}(|x|)\|^{2}_{2}\\&=\||\alpha_{g_{n_{k}}}(x)|-|\alpha_{g_{n_{l}}}(x)|\|^{2}_{2}\\&\leq 2max\{\|\alpha_{g_{n_{k}}}(x)\|_{2},\|\alpha_{g_{n_{l}}}(x)\|_{2}\}\|\alpha_{g_{n_{k}}}(x)-\alpha_{g_{n_{l}}}(x)\|_{2}\\
&\leq 2\|x\|_{2}\|\alpha_{g_{n_{k}}}(x)-\alpha_{g_{n_{l}}}(x)\|_{2}, ~(\text{as } \alpha_g \text{ is an isometry})
\end{align*}
As we take $k,l \rightarrow \infty$ we get the desired result.
\end{proof}
\noindent For $a,b \in L^{2}(\CM,\tau)$,  we define the following: 
\begin{align*}
    max\{a,b\}&=\frac{a+b+|a-b|}{2}\\
    min\{a,b\}&=\frac{a+b-|a-b|}{2}.
\end{align*}
\begin{thm}\label{Density_of_M_c}
Let $x \in L^{2}_{c}(\CM,\tau)$ with $ x = x^*$, then, for all $a\in \mathbb{R}$, $e_{(a,\infty)}(x)\in \CM\cap L^{2}_{c}(\CM,\tau)$. Consequently, $\CM \cap L^{2}_{c}(\CM,\tau)$ is dense in $L_{c}^{2}(\CM,\tau)$. 
\end{thm}
\begin{proof}
We begin by considering the following  sequence of functions
\begin{align*}
f_{n}(t)=min(1,n(max(t-a),0)), ~ \text{ for } t \in \R.
\end{align*}
We note that $(f_{n})_{n\in\mathbb{N}}$ is uniformly bounded above by $1$, and converges pointwise to the function $1_{(a, \infty )}.$ Applying the unbounded operator functional calculus to $x\in L^{2}(\CM,\tau)$, we obtain that the corresponding sequence of operators $(f_{n}(x))_{n \in \mathbb{N}} \subset \CM$ and converge in SOT to $1_{\{s \in \mathbb{R}:s > a\}}(x)=e_{(a,\infty)}(x)$, which is a spectral projection of $x$, i.e.,
\begin{align*}
    \lim_{n \rightarrow \infty}f_{n}(x)\xi= e_{(a,\infty)}(x)\xi
\end{align*}
in $L^{2}$-norm for all $\xi \in L^{2}(\CM, \tau)$, and in particular for $\xi=\widehat{1}$.
By Lemma \ref{modcpt} and definition of max and min, we have that $n(max(x-a),0) \in L^{2}_{c}(\CM,\tau)$ and
\begin{align*}    
f_{n}(x)=min(1,n(max(x-a),0)) \in L^{2}_{c}(\CM,\tau).
\end{align*}
and as a consequence, $e_{(a,\infty)}(x) \in L_c^{2}(\CM,\tau)$ for all $a \in \mathbb{R}$.

\medskip
\noindent Let, $y\in L^2_c(\CM,\tau)$ be self adjoint.
Then $y$ can be written as  $y=y_+-y_-$,   where $y_+=max\{y,0\}$ and $y_-=-min\{y,0\}$. By the definition of $max$ and $min$, $y_+,~y_- \in L^2_c(\CM,\tau)$. Thus, we may assume that $y$ is positive.\\
By the spectral theorem for unbounded operators we have $y=\int_0^\infty \lambda de_\lambda$ where the spectral projections $e_\lambda$ belong to $\CM_c$. 
Let us define $y_n=\int_{\frac{1}{n}}^n\lambda de_\lambda$. Since $y_n$ is a bounded operator and the spectral projections of $y_n$ is in $\CM_c$, therefore, we have  $y_n\in \CM_c$. We also note that 
\begin{align*}
    |y-y_n|^2=\int_0^{\frac{1}{n}} \lambda^2 de_\lambda + \int_n^\infty \lambda^2 de_\lambda \leq y^2.
\end{align*}
Since  $y\in L^2(\CM, \tau)$, $y^2 \in L^1(\CM,\tau)$ and $y_n$ converges to $y$ in measure topology, by Lebesgue Dominated Convergence Theorem, we have that $y_n\to y$ in $\|\cdot\|_2$ as $n \rightarrow \infty$. Hence $\CM \cap L^2_c(\CM, \tau)$ is dense in $L^2_c(\CM, \tau).$
\end{proof}

\noindent Now we are ready to prove the decomposition theorem.
\begin{proof} \textit{of Theorem \ref{JdLG}}: 
We note that  $\CM_{c}$ is a von Neumann subalgebra of $\CM$,  $L^{2}_c(\CM,\tau)$ and $L_{wm}^{2}(\CM,\tau)$ are orthogonal subspaces. Since $\CM$ is a finite von Neumann algebra with f.n trace $\tau$, there exists a f.n conditional expectation  $\CE_{c}: \CM \to \CM_c$ with $ \tau \circ \CE_c = \tau$,  which implements the   decomposition as follows:\\
\noindent Let $a\in \CM$, then $\CE_{c}(a)\widehat{1}\in \CM \cap L_c^{2}(\CM,\tau)$. We need to then show that $(a-\CE_{c}(a))\widehat{1}$ is orthogonal to $L_c^{2}(\CM,\tau)$. Since $\CM \cap L_c^{2}(\CM,\tau)$ is dense in $L^{2}(\CM,\tau)$ in $L^{2}$-norm, it is enough to show the orthogonality to elements of the form $b\widehat{1}$, where $b \in \CM_{c}$.
Indeed, let $b \in \CM_{c}$ as above. We have,
\begin{align*}
 \langle a\widehat{1}-\CE_{c}(a)\widehat{1},b\widehat{1}\rangle_{L^{2}}
=&\tau(b^{*}a-b^{*}\CE_{c}(a))\\
=&\tau(b^{*}a)-\tau(\CE_{c}(b^{*}a)) ~~\text{ (as $b \in \CM_{c}$ )}\\
=&0~ ~\text{ since } \tau \circ \CE_c = \tau.
\end{align*}
therefore, $(a-\CE_{c}(a))\widehat{1} \in \CM \cap L_{wm}^{2}(\CM,\tau)$, $a-\CE_{c}(a) \in \CM_{wm}$, and we are done.

\medskip
\noindent
To show that $\CM \cap L_{wm}^{2}(\CM,\tau)$ is dense in $L^{2}_{wm}(\CM,\tau)$,  we first note that  $L_{c}^{2}(\CM,\tau)$ is the GNS  Hilbert space associated to the von Neumann algebra $\CM_{c}$.  Therefore the conditional expectation $\CE_{c}:\CM \rightarrow \CM_{c}$ extends to a projection $p_{c}:L^{2}(\CM,\tau) \rightarrow L_{c}^{2}(\CM,\tau)$.\\
Let $\xi \in L_{wm}^{2}(\CM,\tau)$, then, by the density of $\CM$ in $L^{2}(\CM,\tau)$, given $\epsilon >0$, there exists $a \in \CM$ such that $\|\xi-a\widehat{1}\|_{2}<\epsilon$, and
\begin{align*}
&\|(I-p_c)(\xi-a\widehat{1})\|_{2}<\epsilon \\
\implies &\|\xi-(I-\CE_c)(a\widehat{1})\|_{2}<\epsilon.
\end{align*}
Since $p_{c}(a\widehat{1})=\CE_{c}(a)\widehat{1}$ when $a\in\CM$, our required element is $a-\CE_{c}(a)$, which is in the weak mixing part by the decomposition.
\end{proof}

\subsection{ Weak mixing Characterization using Reiter Sequence}

A sequence of probability measures $(\mu_n)_n$ on a locally compact group
$G$ is called a \emph{left Reiter sequence} (respectively, a
\emph{right Reiter sequence}) if, for every $s\in G$,
\[
\|\delta_s * \mu_n-\mu_n\|_1 \longrightarrow 0
~~
\left(\text{respectively, }\;
\|\mu_n * \delta_s-\mu_n\|_1 \longrightarrow 0\right)
\]
as $n\to\infty$. We call $(\mu_n)_n$ a \emph{two-sided Reiter sequence}
if it is both a left and a right Reiter sequence.
If $G$ be a l.c.s.c amenable group, then a left
Reiter sequence always exists. In particular, if $(F_n)_n$ is a Følner
sequence in $G$, then
\[
d\mu_n=\frac{\chi_{F_n}}{m(F_n)}\,dm
\]
defines a left Reiter sequence. For further details on Reiter sequences,
we refer the reader to \cite{H_Reiter2000}.
Let $\mu$ be a probability measure on $G$, and let $ \pi:G\longrightarrow\mathcal{U}(\mathcal{H})$
be a strongly continuous unitary representation of $G$ on a Hilbert
space $\mathcal{H}$. Define the operator
\[
T\xi:=\int_G \pi(g)\xi\,d\mu(g),
\qquad \xi\in\mathcal{H}.
\]
By the properties of the Bochner integral, $T$ is a contraction on
$\mathcal{H}$; that is,
\[
\|T\xi\|_{\mathcal{H}}\leq \|\xi\|_{\mathcal{H}},
~~\xi\in\mathcal{H}.
\]
We next establish a mean ergodic theorem for the operators associated
with a Reiter sequence and a characterization of the weak mixing component along a Reiter sequence, which is possibly known in the literature. 
\begin{thm}[Mean Ergodic Theorem for Reiter Sequences]\label{METReiter}
     Let $G$ be a locally compact amenable group, $\pi:G\to \mathcal{U(H})$ a strongly continuous unitary representation on a Hilbert space $\mathcal{H}$, and
$(\mu_n)$ a right Reiter sequence on $G$. Define
$$T_n\xi=\int_G\pi(g)\xi d\mu_n ~~~~~\xi\in \mathcal{H}$$
Then $T_n$ converges strongly to projection onto the fixed point subspace of $\pi$.
 \end{thm}
 \begin{proof}
 Let $\mathcal{H}^\pi=\{\xi \in \mathcal{H}: \pi(g)\xi=\xi \text{ for all }\xi \in \mathcal{H}\}$ and $P$ be the projection onto $\mathcal{H}^\pi.$ Then we have $(\mathcal{H}^\pi)^{\perp}=\overline{span\{\xi-\pi(s)\xi: \xi \in \mathcal{H}, s\in G\}}.$
  For any $\eta=\xi-\pi(s)\xi, s\in G,$ we have 
  \begin{align*}
  \|T_n\eta\|
      &=\|T_n\pi(s)(\xi)-T_n(\xi)\|\\&= \|\int_G\pi(gs)\xi~d\mu_n-\int_G\pi(g)\xi~d\mu_n\|\\
      &=\|\int_G\pi(g)\xi~ d(\mu_n*\delta_s)-\int_G\pi(g)\xi~ d\mu_n\|\\
      &=\|\int_G\pi(g)\xi~ d(\delta_s*\mu_n-\mu_n)\|\\
      &\leq \int_G\|\pi(g)\xi\|~ d|\delta_s*\mu_n-\mu_n|\\
      &=\|\xi\|\|\mu_n*\delta_s-\mu_n\|_1
  \end{align*}
Using the right Reiter property of $(\mu_n)_n,$ $\|T_n\eta\| \to 0.$Hence the density of $span\{\xi-\pi(s)\xi: \xi \in \mathcal{H}, s\in G\} $ and uniform boundedness of $(T_n)$ implies $T_n\eta \to 0$ for every $\eta \in (\mathcal{H}^{\pi})^{\perp}.$
Then every $\xi \in \mathcal{H}$ decomposes as $$\xi=P\xi+(\xi-P\xi)$$
with $P\xi\in \mathcal{H}^\pi$ and $\xi-P\xi\in (\mathcal{H}^\pi)^{\perp}.$ Hence
$$T_n\xi=T_n(P\xi)+T_n(\xi-P\xi)
\longrightarrow
P\xi$$
proving that $T_n\xrightarrow{s}P.$
 \end{proof}

\begin{thm}\label{thm:weak-mixing-Reiter}
Let $G$ be an amenable locally compact group, $\pi:G\to \mathcal{U(H)}$ a strongly
continuous unitary representation, and $(\mu_n)$ a left Reiter sequence on
$G$. Suppose $ \xi \in \CH_{ wm} $, it follows that
$$ \int_G |\langle \pi(g)\xi,\eta\rangle|\,d\mu_n(g)
    \longrightarrow 0, ~~ \text{ for all } \eta \in \CH.$$
    
\end{thm}
\begin{proof}
 Consider the conjugate Hilbert space $\overline{ \CH} $ of $\CH$ and the representation $ \overline{\pi} : G \to \CU(\overline{\CH})$ . Consider the tensor product representation
$\pi\otimes\overline{\pi}$ on $\mathcal{H}\otimes\overline{\mathcal{H}}$. Since
\[
\left\langle
(\pi\otimes\overline{\pi})(g)
(\xi\otimes\overline{\xi}),
\eta\otimes\overline{\eta}
\right\rangle
=
|\langle\pi(g)\xi,\eta\rangle|^2,
\]
by Theorem \ref{METReiter} we have,
\begin{align*}
    \int_{G}|\langle\pi(g)\xi,\eta\rangle|^2d\mu_{n}(g) \rightarrow\langle P_{\pi \otimes \overline{\pi}}(\xi \otimes \overline{\xi}),\eta \otimes \overline{\eta}\rangle \text{ as $n \rightarrow \infty$},
\end{align*}
where $P_{\pi \otimes \overline{\pi}}$ denotes the orthogonal projection onto the fixed point subspace of $\pi\otimes \overline{\pi}.$

\medskip
\noindent
Then it follows that  $P_{\pi \otimes \overline{\pi}}(\xi\otimes\overline{\xi})=0.$
To show this, we consider the unitary  $ \Phi:\CH \otimes \overline{\CH} \rightarrow \CH\CS(\CH)$ defined by 
\begin{align*}
\Phi(\xi \otimes \overline{\eta}) = \omega_{\eta,\xi}(\cdot)=\langle\cdot,\eta\rangle\xi,
\end{align*}
where $\CH \CS(\CH)$ is the space of Hilbert Schmidt operators.
Therefore, $\pi \otimes \overline{\pi}$ can be seen as a representation on $\CH \CS(\CH)$   and the action is by conjugation:
\begin{align*}
(\pi \otimes \overline{\pi})(g)(T)=\pi(g)T\pi(g)^{*}, \text{ for all } T \in \CH \CS(\CH).
\end{align*}
Let $(\CH \CS(\CH))^{G}=\{T \in \CH \CS(\CH): \pi(g)T=T\pi(g) \text{ for all $g \in G$}\}.$
We note that the unitary $\Phi:\CH \otimes\overline{\CH} \rightarrow\CH\CS(\CH)$ satisfies $\Phi((\pi \otimes \overline{\pi})^{G})=(\CH \CS(\CH))^{G}$.

\medskip 
\noindent
Now we note that  $ \omega_{\xi,\xi} \perp (\CH \CS(\CH))^{G}.$ Indeed, first observe that 
\begin{align*}
\text{Tr}((\omega_{\xi,\xi})^{*}T)=\langle T^{*}\xi,\xi \rangle, ~~ \text{ for all } T \in (\CH\CS(\CH))^{G},
\end{align*}
for all $T \in (\CH\CS(\CH))^{G}$, where  $\text{Tr}$ is the trace on $\CB(\CH)$. 
Equivalently, we need to show that $ \xi \in \text{ker}(T^*) $  for  all $ T \in (\CH\CS(\CH))^{G} $.  
Since $\text{ker}(T^*)=\text{ker}(TT^{*})$, and applying the spectral theorem for compact self-adjoint operators to the operator $TT^{*}$, we obtain
\begin{align*}
TT^{*}(\eta)=\sum_{n=1}^{\infty}\lambda_{n}\langle\eta,e_{n}\rangle e_{n}~~ \text{ for } \eta \in \CH,
\end{align*}
where \((e_n)\) is an orthonormal family of vectors in $\CH$ and \(\lambda_n>0\). Therefore, 
$ \text{ker}(TT^*)^\perp
=
\overline{\operatorname{span}}\{e_n:n\in\mathbb N\}.$
Since \(T\) commutes with \(\pi(g)\), so does \(TT^*\). Thus
\[
TT^*\pi(g)e_n
=
\pi(g)TT^*e_n
=
\lambda_n\pi(g)e_n,
~~ g\in G.
\]
Therefore, for each $n \in \N$,    $ \{ \pi(g)e_n : g \in G  \}  $ contained in the eigenspace for $ \lambda_n$, which is finite dimensional. Thus,   we have $ \text{ker}(TT^*)^\perp\subseteq\mathcal H_c.$
Taking orthogonal complements gives
\[
\mathcal H_{wm}
=
\mathcal H_c^\perp
\subseteq
\text{ker}(TT^*)
=
\text{ker}(T^*).
\]
Since $\xi \in \CH_{wm}$, $ \text{Tr}((\omega_{\xi,\xi})^{*}T)=0, ~ \text{ for all } T \in (\CH\CS(\CH))^{G}.$
Therefore, $ \xi \otimes \overline{\xi} \perp (\pi\otimes \overline{\pi})^{G}$ and consequently, it follows that  $  P_{\pi \otimes \overline{\pi}}(\xi\otimes\overline{\xi})=0.$ Then the conclusion of the theorem follows from Cauchy-Schwarz  inequality.

\end{proof}

\section{ \bf Non-commutative Wiener-Winter theorem } 

 In this section, we establish a noncommutative Wiener-Wintner theorem. Let $\CM$ be a finite von Neumann algebra equipped with a faithful normal trace $\tau$, and let $G$ be a locally compact second countable group. Suppose that $(\CM,G,\alpha)$ is a $\tau$-preserving dynamical system. 

\medskip
\noindent
Let $\mathrm{Rep}^d_u(G)$ denote the collection of all finite-dimensional unitary representations
\[
\pi : G \to \CU_d(\C),
\]
where $\CU_d(\C)$ denotes the group of all unitary matrices in $M_d(\C)$ for some $d \in \N$. Further, let $C_1(\CU_d(\C))$
denote the collection of all continuous functions on $\CU_d(\C)$ with norm equal to $1$. Define
\[
\S := \mathrm{Rep}^d_u(G)\times \CU_d(\C)\times C_1(\CU_d(\C)).
\]
For $x \in L^p(\CM,\tau)$, $1<p<\infty$, and $\Gamma=(\pi,u,\phi)\in \S$, consider the averages
\begin{align*}
\A_n(x)
&=
\frac{1}{m(F_n)}
\int_{F_n}\alpha_g(x)\, d\mu(g),\\
\A_n(x,\Gamma)
&=
\frac{1}{m(F_n)}
\int_{F_n}\alpha_g(x)\,\phi(\pi(g))u \, d\mu(g).
\end{align*}
In this setting, we aim to establish the b.a.u  convergence of the averages  $( \A_n(x,\Gamma))$ for $ x \in L^p(\CM, \tau)$ with $  1 \leq p < \infty$, 
uniformly over $\Gamma \in \mathcal{S}$; that is, the projection appearing in the definition of bilaterally almost uniform (b.a.u) convergence can be chosen independently of $\Gamma$.

\medskip
\noindent
Throughout this section, we assume that $(F_n)$ is an admissible F\o lner sequence, unless stated otherwise. We begin by recalling the following theorem from \cite{cadilhac2022noncommutative}, which will be used in the sequel.
\begin{thm}\label{Cadilac-Wang}
Let $(F_n)$ be an admissible F\o lner sequence on $G$. Let $\CM$ be a semifinite von Neumann algebra equipped with a faithful normal trace $\tau$, and let $(\CM,G,\alpha)$ be a $\tau$-preserving dynamical system. Then, for every $x\in L^1(\CM,\tau)$ and $\lambda>0$, there exists a projection $e\in \CM$ such that
\[
\sup_{n\geq 1}\|e\A_n(x)e\|_\infty \leq \lambda, ~~
 \text{  and   } ~~
\lambda\,\tau(1-e)\lesssim \|x\|_1.
\]
Moreover, for every $x\in L^p(\CM,\tau)_+$ with $1<p<\infty$, there exists $a\in L^p(\CM,\tau)_+$ satisfying
\[
\A_n(x)\leq a,~~ \text{ for all } n\geq 1, \text{  and  } \|a\|_p \lesssim \|x\|_p.
\]
\end{thm}
In \cite{cadilhac_noncommutative_2019}, Cadilhac and Wang further proved that, for $x\in L^p(\CM,\tau)$, the averages $\A_n(x)$ converge bilaterally almost uniformly whenever $1\leq p<2$, and converge almost uniformly whenever $2\leq p<\infty$.

For convenience, we record the following consequence, which will be used in later results. We also include a short proof for completeness.

\begin{lem}\label{Fix_point_density}
  $(\CM, G,\alpha)$ be as above. Then we have 
  \begin{align*}
      \overline{\CM^G}^{\|\cdot\|_2} = L^2(\CM,\tau)^G,
  \end{align*}
\end{lem}
\noindent where $\CM^{G}$ is the fixed point subalgebra of the action $\alpha:G \rightarrow Aut(\CM)$ and $L^{2}(\CM,\tau)^{G}$ is the fixed point subspace of the unitary representation of $G$ obtained by extending $\alpha_{g}$ to $L^{2}(\CM,\tau)$ for all $g \in G$.
\begin{proof}
Inclusion $\subseteq$ is immediate. Conversely, suppose $x \in L^2(\CM, \tau)^G$. Since $\alpha_g(x^*)=x^*$, we can assume that $x$ is self-adjoint. Corollary \ref{mod_preserving} implies that we may assume that $x$ is positive self-adjoint. By the spectral theorem we have $$x=\int_0^\infty \lambda de_\lambda ~~ \text{and} ~~ e_\lambda \in \CM.$$
   Since $\alpha_g$ is normal, for all $g \in G$,
   $$\alpha_g(x)=\int_0^\infty \lambda d \alpha_g (e_\lambda).$$
  Hence $\alpha_g(x)=x$ and the uniqueness of spectral projections implies $e_\lambda \in \CM^G$. Since $\CM^G$ is a von Neumann algebra, we use an argument similar to the proof of Theorem \ref{Density_of_M_c} to conclude that $x \in \overline{\CM^G}^{\|\cdot\|_2}$ and this completes the proof.    
\end{proof}

\begin{prop}\label{Mean_ergodic_thm} Let  $\CM$ be a finite von Neumann algebra with f.n trace $\tau$ and 
 $G$ be a l.c.s.c. amenable  group with (left) Haar measure $m$ and  $(F_n)$ be an admissible F\o lner sequence. Moreover, we assume that the action is ergodic (i.e. $\CM^G = \mathbb{C}1$).
 Then for every $x \in L^2(\CM,\tau)$,
\[
\A_n(x) \longrightarrow \tau(x)1 \quad \text{almost uniformly (a.u.)} \text{ and in } \norm{ \cdot }_2.
\]
\end{prop}

\begin{proof} 
By the mean ergodic theorem for amenable group actions,
\[
\A_n(x) \xrightarrow{ \norm{\cdot}_2} \bar{x} \quad \text{in } L^2(\CM, \tau),
\]
where $\bar{x}$ is the orthogonal projection of $x$ onto the invariant subspace
\[
L^2(\CM,\tau)^G = \{ y \in L^2(\CM, \tau): \alpha_g(y)=y ,~\text{for all } g \in G \}.
\]
For $h \in G$,
\[
\alpha_h(\A_n(x)) 
= \frac{1}{m(F_n)} \int_{F_n} \alpha_{hg}(x)\, dm(g)
= \frac{1}{m(F_n)} \int_{hF_n} \alpha_{g}(x)\, dm(g).
\]
Hence
\begin{align*}
 \norm{ \alpha_h(\A_n(x)) - \A_n(x)}_2
&= \norm{\frac{1}{m(F_n)} \left( \int_{hF_n} - \int_{F_n} \right)\alpha_g(x)\, dm(g)}_2\\
&= \norm{  \frac{1}{m(F_n)} \left( \int_{(hF_n\setminus F_n)\cup (F_n\cap hF_n)} - \int_{(F_n\setminus hF_n) \cup ( F_n \cap hF_n)} \right)\alpha_g(x)\, dm(g)}_2\\
&\leq \|x\|_2\frac{m(hF_n \triangle F_n)}{m(F_n)} \xrightarrow{ n \to \infty } 0
\end{align*}
Therefore, we have $ \alpha_h(\bar x) =\bar x $ for all $ h \in G$.
Thus $\bar{x} \in L^2(\CM,\tau)^G$. Ergodicity of $\alpha$ (see Lemma \ref{Fix_point_density})  implies that $\bar{x}=c(x)1$. Further, since the action is  $\tau$-invariant,
\[
\tau(\A_n(x)) 
= \frac{1}{m(F_n)} \int_{F_n} \tau(\alpha_g(x))\, dm(g)
= \tau(x).
\]
Taking limits, $ \tau(\bar{x}) = \tau(x).$
Since $\bar{x} = c(x) 1$, it follows that $ c(x) = \tau(x),$ 
hence $ \bar{x} = \tau(x)\, 1.$
 We also know that the average $\A_n(x)$ converges in a.u. to some $\tilde{x} \in L^2(\CM,\tau)$. Almost uniform convergence implies convergence in measure, and since $L^2$-convergence also implies convergence in measure, uniqueness of the limit gives $ \tilde{x} = \bar{x}.$ Thus, $ \A_n(x) \longrightarrow \tau(x)1, ~~\text{ in a.u}.$
\end{proof}

We begin by recalling the following definition.
\begin{defn}
      Let $(X, \|\cdot\|)$ be a normed space. A sequence $a_n : X \to L^0(\CM, \tau)$ of additive maps is called \textit{bilaterally uniformly equicontinuous in measure (b.u.e.m.)} at $0 \in X$ if for every $\varepsilon > 0$, $\delta > 0$ there exists $\gamma > 0$ such that for every $x \in X$ with $\|x\| < \gamma$ there is $e_x \in P(\mathcal{M})$ for which
\[
\tau(e_x^\perp) \leq \varepsilon \quad \text{and} \quad \sup_n \|e_x a_n(x) e_x\|_\infty \leq \delta.
\]
  \end{defn}

Now we first derive the following theorem from Theorem~\ref{Cadilac-Wang}.
\begin{thm}\label{maximal_inequality}
    Let $ x \in L^p(\CM, \tau)$ for $ 1\leq p  < \infty $, and $\lambda > 0$, then there exists a projection $e \in \CP(\CM)$ such that 
    $$ \sup_n \norm{ e\A_n( x) e } \lesssim \lambda ~~\text{ and } \tau( 1-e)^{ \frac{1}{p}}  \lesssim\frac{ \norm{ x}_p}{ \lambda}.$$
 Moreover, for every $\varepsilon > 0$ and $\delta > 0$, there exists $\gamma > 0$ such that for every $x \in L^p( \CM, \tau)$ with $\|x\|_p < \gamma$, there exists $e_x \in \CP(\CM)$ for which
\[
\tau\bigl(e_x^\perp\bigr) \leq \varepsilon
\quad \text{and} \quad
\sup_n \bigl\| e_x \A_n(x, \Gamma) e_x \bigr\|_\infty \lesssim \delta,  ~~\text{ for all } \Gamma \in \S.
\]
\end{thm}
\begin{proof}
The first part of the theorem follows from standard argument. Indeed, for $p=1$, nothing to show  and for $ p> 1$, let $e = \chi_{ [0, \lambda]}(x) $, then it is straightforward to check the following

$$ 
\sup_n \norm{ e\A_n( x) e } \lesssim \lambda ~~\text{ and } \tau( 1-e)^{ \frac{1}{p}}  \lesssim\frac{ \norm{ x}_p}{ \lambda}.
$$
Now let $ x \in L^p(\CM, \tau)$ and note that 
     $x$ can be written as $x=x_1-x_2+i(x_3-x_4)$ where $x_j \geq 0$ and $\|x_j\|_p \leq \|x\|_p $ for all $j\in \{1,2,3,4\}$.  By the first part of the theorem 
    there exists $e_j \in \CP(\CM)$ such that 
    \begin{align*}
        \displaystyle\sup_{n\geq1}\|e_j\A_n(x_j)e_j\|_\infty \leq \lambda ~~\text{ and }~~\ \tau(1-e_j)\lesssim \frac{\|x_j\|_p^p}{\lambda^p} 
    \end{align*}
  Hence we have 
    \begin{align*}
        \sup_{n\geq 1} \|e_j\A_n(x_j)e_j\|^p \leq  \sup_{n\geq 1} \|e_j\A_n(x_j^p)e_j\| \leq \lambda^p
        \implies  \sup_{n\geq 1} \|e_j\A_n(x_j)e_j\| \leq \lambda.
    \end{align*}
Define $e:=\wedge_{j=1}^4e_j$, then we note that  $\tau(1-e)\lesssim \frac{\|x\|_p^p}{\lambda^p}$ and  $\displaystyle\sup_{n\geq1}\|e\A_n(x_j)e\| \leq \lambda$ for all $ j \in \N$.
    Now consider the following;
    \begin{align*}
        \A_n^{(R)}(x_j,\Gamma):= \frac{1}{m(F_n)}\int_{F_n}\alpha_g(x_j)Re(\phi(\pi(g)u)) dm(g) + \A_n(x_j, \Gamma)\\
        \A_n^{(I)}(x_j,\Gamma):= \frac{1}{m(F_n)}\int_{F_n}\alpha_g(x_j)Im(\phi(\pi(g)u)) dm(g) + \A_n(x_j, \Gamma).
    \end{align*}
    Since $1+Re(\phi(v)) \leq 1+\|\phi\|_\infty \leq 2$ and $x_j \geq 0$,  we have  $ \|e\A_n^{(R)}(x_j,\Gamma)e\| \leq 2 \|e\A_n(x_j)e\|$
  and this implies  $  \|e\A_n^{(R)}(x_j,\Gamma)e\| \lesssim\lambda.$
    Similarly we get $  \|e\A_n^{(I)}(x_j,\Gamma)e\| \lesssim \lambda.$
Hence,  we have
\begin{align*}
    \sup_{n\geq 1}\|e\A_n(x_j, \Gamma)e\|=\sup_{n\geq 1}\|e(\A_n^{(R)}(x_j,\Gamma) +i\A_n^{(I)}(x_j,\Gamma) - \A_n(x_j)- i\A_n(x_j))e\| \lesssim \lambda.
\end{align*}
This proves that
    \begin{align*}
      \sup_{n\geq 1} \| e\A_n(x,\Gamma)e\|=\sup_{n\geq 1}\|e(\A_n(x_1,\Gamma)-\A_n(x_2,\Gamma)+i\A_n(x_3,\Gamma)-i\A_n(x_4,\Gamma))e\| \lesssim \lambda.
    \end{align*}
By the first part there exists $C>0, e_x\in \CP(\CM)$ such that
    \begin{align*}
      \tau(1-e_x)\leq C^p\frac{\|x\|_p^p}{\lambda^p} ~~\text{and} ~~ \sup_{n\geq1}\|e
    _x \A_n(x,\Gamma)e_x\| \lesssim \lambda, ~~ \text{for all}~\lambda>0.
    \end{align*}
Given $\epsilon, \delta >0$, we can choose $0<\gamma <\epsilon^{1/p} \frac{\delta}{C}$.
Then for all $x \in L^p(\CM, \tau)$ with $\|x\|_p< \gamma$, there exists $e_x \in \CP(\CM)$ for which we have 
\begin{align*}
    \tau(1-e_x)\leq \epsilon ~~\text{and}~~ \sup_{n\geq 1}\|e_x\A_n(x,\Gamma)e_x\| \lesssim \delta.
\end{align*}
This completes the proof.
\end{proof}

\begin{rem}
    Note that the projection $e_x$ in the proof of Theorem \ref{maximal_inequality} does not depend on $\Gamma \in \S$. Consequently $\A_n(x,\Gamma)$ is b.u.e.m at $0$.
\end{rem}
Analogous to the classical case, the following Banach principle provides a connection between maximal inequalities and bilateral almost uniform 
(b.a.u.) convergence, and the proof can be found in \cite[Theorem 2.1]{Lit2014}.

\begin{thm}\label{Banach_Principle} Let $ 1 \leq p < \infty $, then the following set 
$$ 
\CC = \{ x \in L^p(\CM, \tau) : \A_n(x, \Gamma) \text{ converges in b.a.u uniformly in } \Gamma \in \S \}
$$
is a closed set in $ L^p(\CM, \tau)$. 
\end{thm}
Thus, to prove the b.a.u convergence of the sequence $ ( \A_n(x, \Gamma))$ for $ x \in L^p(\CM, x )$ for $ 1 \leq p < \infty $ uniformly for all $ \Gamma \in \S$, it remains to verify such convergence on a dense set of $ L^p(\CM, \tau)$.

\medskip
\noindent
The van der Corput is often used to establish the convergence of various ergodic averages, both in classical and noncommutative settings. 
Here we state and prove a  noncommutative version of this lemma for a bounded function $f:G \rightarrow \CM$, where $G$ is an amenable group and $\CM$ is a von Neumann algebra.
\begin{lem}[A van der Corput inequality for amenable groups]\label{Lemma_vand}
Let $G$ be a l.c.s.c. amenable group equipped with a right Haar measure $m$, and let $(F_n)$ be a right F\o lner sequence in $G$ and $\CM$ be von Neumann algebra. Let $f : G \to \CM$ be a bounded measurable function. Then for every $n, k \in \mathbb{N}$, the following estimate holds:
\begin{align}\label{Vand_inequality}
\norm{
\frac{1}{m(F_n)} \int_{F_n} f(g)\, dm(g)
}^2 &\leq 
\frac{1}{m(F_k)^2} \int_{F_k} \int_{F_k}
\norm{ \left( \frac{1}{m(F_n)} \int_{F_n}
f( g)^* ~ f( gh_{1}^{-1}h_2) dm(g) \right) }
dm(h_1)\, dm(h_2) \\
&~~~+ 3\|f\|_\infty^2 \sup_{h \in F_k}
\frac{m(F_nh \triangle F_n)}{m(F_n)}  +   \|f\|^2_\infty \left( \sup_{h \in F_k}
\frac{m(F_nh \triangle F_n)}{m(F_n)} \right)^2.
\end{align}

\end{lem}

\begin{proof}
We fix a  $h \in F_k$. We first estimate
\begin{align*}
\norm{ \frac{1}{m(F_n)} \int_{F_n} f(gh)\, dm(g) }
&= \norm{ \frac{1}{m(F_n)} \int_{F_nh} f(g)\, dm(g) } \\
&=\norm{ \frac{1}{m(F_n)} \int_{F_n} f(g) ~ dm(g) +  \frac{1}{m(F_n)} \int_{F_nh} f(g) ~ dm(g)  - \frac{1}{m(F_n)} \int_{F_n} f(g) ~ dm(g)}\\
&\leq \norm{ \frac{1}{m(F_n)} \int_{F_n} f(g)\, dm(g) }\\& \qquad + \norm{ \frac{1}{m(F_n)} \int_{F_nh \cap F_n} f(g)\, dm(g)
     + \frac{1}{m(F_n)} \int_{F_nh \setminus F_n} f(g)\, dm(g) }\\
&\leq 
 \norm{ \frac{1}{m(F_n)} \int_{F_n} f(g)\, dm(g) }  + \|f\|_\infty \frac{m(F_nh \triangle F_n)}{m(F_n)} \\
 &\leq 
 \norm{ \frac{1}{m(F_n)} \int_{F_n} f(g)\, dm(g) }  + \sup_{ h \in F_k}\|f\|_\infty \frac{m(F_nh \triangle F_n)}{m(F_n)} .
\end{align*}
Thus, we obtain 
\begin{equation}\label{4.1}
\norm{ \frac{1}{m(F_n)} \int_{F_n} f(gh)\, dm(g) } \leq \norm{ \frac{1}{m(F_n)} \int_{F_n} f(g)\, dm(g) }  + \sup_{ h \in F_k}\|f\|_\infty \frac{m(F_nh \triangle F_n)}{m(F_n)} .
\end{equation}

\medskip
\noindent
Next, we compare the average of $f$ with its averaged translates:
\begin{align*}
\norm{ \frac{1}{m(F_n)} \int_{F_n} f(g)\, dm(g) }
&\leq \norm{  \frac{1}{m(F_n)} \int_{F_n} f(g)\, dm(g)
- \frac{1}{m(F_n)m(F_k)} 
\int_{F_n} \int_{F_k} f(gh)\, dm(h)\, dm(g) } \\
& \qquad+ \norm{  \frac{1}{m(F_n)m(F_k)} 
\int_{F_n} \int_{F_k} f(gh)\, dm(h)\, dm(g) }\\
& \leq  \|f\|_\infty \sup_{h \in F_k}
\frac{m(F_nh \triangle F_n)}{m(F_n)} + \norm{  \frac{1}{m(F_n)m(F_k)} 
\int_{F_n} \int_{F_k} f(gh)\, dm(h)\, dm(g) }.
\end{align*}
Therefore, by squaring both sides, we obtain
\begin{align*}
\norm{ \frac{1}{m(F_n)} \int_{F_n} f(g)\, dm(g) }^2
&\leq \norm{ \frac{1}{m(F_n)m(F_k)} 
\int_{F_n} \int_{F_k} f(gh)\, dm(h)\, dm(g)}^2 \\& \qquad+ 2 \|f\|_\infty \left( \sup_{h \in F_k}
\frac{m(F_nh \triangle F_n)}{m(F_n)} \right)\norm{ \frac{1}{m(F_n)m(F_k)} 
\int_{F_n} \int_{F_k} f(gh)\, dm(h)\, dm(g)} \\& \qquad+  \|f\|^2_\infty \left( \sup_{h \in F_k}
\frac{m(F_nh \triangle F_n)}{m(F_n)} \right)^2\\
& \leq  \norm{ \frac{1}{m(F_n)m(F_k)} 
\int_{F_n} \int_{F_k} f(gh)\, dm(h)\, dm(g)}^2 \\&\qquad+ 2 \norm{f}^{2}_{\infty} \left( \sup_{h \in F_k}
\frac{m(F_nh \triangle F_n)}{m(F_n)} \right) \\& \qquad +  \|f\|^2_\infty \left( \sup_{h \in F_k}
\frac{m(F_nh \triangle F_n)}{m(F_n)} \right)^2.
\end{align*}

\medskip
\noindent
Now we observe that 
\begin{align*}
&\norm{
\frac{1}{m(F_n)} \int_{F_n}
\frac{1}{m(F_k)} \int_{F_k} f(gh)\, dm(h)\, dm(g)
}^2\\
&\leq \norm{  \frac{1}{m(F_n)} \int_{F_n} \abs{
\frac{1}{m(F_k)} \int_{F_k} f(gh)\, dm(h)}^2 dm(g) } ~~\text{ by Jensen Inequality}\\
&= \norm{\frac{1}{m(F_n)} \int_{F_n}
\frac{1}{m(F_k)^2} \int_{F_k} \int_{F_k}
f( gh_1)^* f( gh_2) ~ dm(h_1)\, dm(h_2)\, dm(g) } \\
&= \frac{1}{m(F_k)^2} \int_{F_k} \int_{F_k}
\norm{ \left( \frac{1}{m(F_n)} \int_{F_n}
f( gh_1)^* ~ f( gh_2) dm(g) \right) }
dm(h_1)\, dm(h_2)\\
&= \frac{1}{m(F_k)^2} \int_{F_k} \int_{F_k}
\norm{ \left( \frac{1}{m(F_n)} \int_{F_nh_1}
f( g)^* ~ f( gh_{1}^{-1}h_2) dm(g) \right) }
dm(h_1)\, dm(h_2)\\
&\leq  \frac{1}{m(F_k)^2} \int_{F_k} \int_{F_k}
\norm{ \left( \frac{1}{m(F_n)} \int_{F_n}
f( g)^* ~ f(gh_{1}^{-1}h_2) dm(g) \right) }
dm(h_1)\, dm(h_2) + \|f\|_\infty^2 \sup_{h \in F_k}
\frac{m(F_nh \triangle F_n)}{m(F_n)}.
\end{align*}
Combining these estimates,  we conclude that
\begin{align*}
\norm{
\frac{1}{m(F_n)} \int_{F_n} f(g)\, dm(g)
}^2 &\leq 
\frac{1}{m(F_k)^2} \int_{F_k} \int_{F_k}
\norm{ \left( \frac{1}{m(F_n)} \int_{F_n}
f( g)^* ~ f( gh_{1}^{-1}h_2) dm(g) \right) }
dm(h_1)\, dm(h_2) \\
&~~~+ 3\|f\|_\infty^2 \sup_{h \in F_k}
\frac{m(F_nh \triangle F_n)}{m(F_n)}  +   \|f\|^2_\infty \left( \sup_{h \in F_k}
\frac{m(F_nh \triangle F_n)}{m(F_n)} \right)^2.
\end{align*}

\end{proof}
\begin{rem}
    The above Lemma can also be proved for left amenable groups using left F\o lner sequences. However, in our main theorem, we assume that $G$ is unimodular. Therefore, this difference will not cause any problem.
\end{rem}
\noindent We are now in a position to identify a dense subset on which the averages  $ \A_n(x,\Gamma)$
converge a.u., uniformly over all $\Gamma \in \S$. More precisely, we shall show that for every
\[
x \in ( L_c^2(\CM,\tau)) + (\CM \cap L_{wm}^2(\CM,\tau)),
\]
the averages $ \A_n(x,\Gamma)$
converge a.u. and uniformly in $\Gamma \in \S$. We begin with the following lemma.
\begin{lem}
$\alpha_g(L^{2}_{wm}(\CM,\tau)) \subset L^{2}_{wm}(\CM,\tau)$, for all $  g\in G $.
\end{lem}
\begin{proof}
    Proof follows from the fact that $L_{c}^{2}(\CM,\tau)$ is $\alpha_g$-invariant for all $  g\in G $.
\end{proof}
\noindent We refer to \cite[Theorem 15.14]{Eisner2015} or \cite[Theorem 3.5]{EliGlasner}, for the following characterization of the compact part $L^2_c(\CM,\tau)$.
\begin{lem}\label{cpt_characterization}
  Let $(\CM,G,\alpha)$ be as above. Then we have,
  \begin{align*}
      L^2_c( \CM, \tau)=\overline{\text{span} \{\text{ $\cup \CK$: $\CK$ is a finite-dimensional, $\alpha$-invariant subspace of $L^2(\CM,\tau)$}\}}^{\|.\|_{2}}.
  \end{align*}
\end{lem}
\begin{thm}\label{Cpt_conv}
    Let $x \in L^2_c( \CM, \tau)$ and $\xi, \eta \in \C^d$,  then the following weighted  average 
    \begin{align}\label{Avg_1}
        \B_n(x):= \frac{1}{m(F_n)} \int_{F_n} \alpha_g(x) \langle \pi(g)(\xi), \eta \rangle dm(g)
    \end{align}
    converges in $a.u.$
\end{thm}
\begin{proof}
  Let $\CK$ be a finite dimensional invariant subspace of $L^2(\CM, \tau)$ and suppose 
 $\{u_1,u_2,\cdots,u_k\}$ is an orthonormal basis for $\CK$. Let  $x\in \CK$,  then we can write   $x=\displaystyle\sum_{i=1}^kc_iu_i$, for constants $ c_i \in \C$ for $i = 1, 2, \cdots, k$. Therefore, we have 
 $$\alpha_g(x)= \displaystyle\sum_{i=1}^kc_i\alpha_g(u_i), ~ \text{ for all } g \in G.
 $$
 Hence, it is enough to prove the convergence of the average (\ref{Avg_1}) for $x=u_i$. Now $\alpha_g(u_i)=\displaystyle \sum_{j=1}^k \langle \alpha_g(u_i),u_j\rangle u_j.$
 We have \begin{align*}
     \B_n(u_i)&= \frac{1}{m(F_n)} \int_{F_n} \alpha_g(u_i) \langle \pi(g)(\xi), \eta \rangle dm(g)\\
     &= \frac{1}{m(F_n)} \int_{F_n} \displaystyle \sum_{j=1}^k \langle \alpha_g(u_i),u_j\rangle u_j \langle \pi(g)(\xi), \eta \rangle dm(g)\\
     &= \displaystyle \sum_{j=1}^k \frac{1}{m(F_n)} \int_{F_n} \langle \alpha_g(u_i), u_j \rangle \langle \pi(g)(\xi), \eta\rangle dm(g)u_j\\
     &=\displaystyle \sum_{j=1}^k \frac{1}{m(F_n)} \int_{F_n} \langle \alpha_g(u_i)\otimes \pi(g)(\xi), u_j\otimes \eta \rangle dm(g)u_j\\
     &=\displaystyle \sum_{j=1}^k \langle \frac{1}{m(F_n)} \int_{F_n} \alpha_g(u_i)\otimes \pi(g)(\xi) dm(g), u_j\otimes \eta \rangle u_j.
 \end{align*}
 Now by the mean ergodic theorem, the averages
 \begin{align*} \frac{1}{m(F_n)} \int_{F_n} \alpha_g(u_i)\otimes \pi(g)(\xi) dm(g)=\frac{1}{m(F_n)} \int_{F_n} \alpha_g\otimes \pi(g)(u_i\otimes \xi) dm(g)
 \end{align*}
 converge in $\|.\|_2$. Hence, the averages
 \begin{align*}
 \langle \frac{1}{m(F_n)} \int_{F_n} \alpha_g(u_i)\otimes \pi(g)(\xi) dm(g), u_j\otimes \eta \rangle
 \end{align*}
 converge for each $j \in \{1,2, \cdots, k\}$. For each $j$, there exists a projection $e_j$ with $\tau(e_j^\perp)< \frac{\epsilon}{k}$ and $u_je_j \in \CM$. Taking $e=\displaystyle\wedge_{j=1}^k e_j$, we get
 \begin{align*}
\tau(e^\perp) < \displaystyle \sum_{j=1}^k \tau(e_j^\perp) < \epsilon
 \end{align*}
 and $u_je \in \CM$ for each $j \in \{1,2,\cdots, k\}$. Hence  $\B_n(u_i)e$ converges in $\|.\|_\infty$. By the Theorem \ref{Banach_Principle} and Lemma \ref{cpt_characterization}, we conclude the convergence of $\B_n(\cdot)$ for all $x \in L^{2}_c(\CM,\tau).$
\end{proof}

Let $G$ be a l.c.s.c amenable group with the F\o lner sequnce $(F_n)$ and $m$ be a left Haar measure on $G.$
Consider the left Reiter sequence $(\mu_n)_n$ where $d\mu_n:= \frac{1}{m(F_n)}\chi_{F_n}dm$
Define $$\nu_n=\widetilde{\mu_n}*\mu_n.$$ where $\widetilde{\mu_n}(E):= \mu_n(E^{-1}), E \subseteq G.$ We need the following lemma to prove the convergence for weak-mixing part.
 \begin{lem}\label{lemma_Reiter}
   Let $G$ be a l.c.s.c amenable unimodular group with two-sided F\o lner sequence $(F_n)$ and $(\mathcal{M},\tau)$ be a finite von Neumann algebra. $\alpha: G \to Aut(\mathcal{M})$ be a trace preserving weakly mixing action. Suppose $x \in \mathcal{M}$ satisfies $\tau(x)=0.$ Then we have 
   $$\frac{1}{m(F_n)^2}
\int_{F_n}\int_{F_n}
\left|
\tau\bigl(x^*\alpha_{h_1^{-1}h_2}(x)\bigr)
\right|
\,dm(h_1)\,dm(h_2)
\longrightarrow0.$$
 \end{lem}
 \begin{proof}
  By the definition of $\widetilde{\mu_n}$ and convolution of measures we have 
     \begin{align}
         \frac{1}{m(F_n)^2}
\int_{F_n}\int_{F_n}
\left|
\tau\bigl(x^*\alpha_{h_1^{-1}h_2}(x)\bigr)
\right|
\,dm(h_1)\,dm(h_2)= 
\int_G
\left|
\tau\bigl(x^*\alpha_{g}(x)\bigr)
\right|
\,d(\widetilde{\mu_n}*\mu_n)(g).
     \end{align}
Applying the Theorem \ref{thm:weak-mixing-Reiter} it suffices to show that $\widetilde{\mu_n}*\mu_n$ is a two-sided Reiter sequence. From the unimodularity of $G$ it is immediate that $\widetilde{\mu_n}*\mu_n$ is a right Reiter sequence. To show that it is left Reiter,
for all $s \in G,$
\begin{align*}
    \|\delta_s*(\widetilde{\mu_n}*\mu_n)-\widetilde{\mu_n}*\mu_n\|_1 &= \|(\delta_s*\widetilde{\mu_n})*\mu_n-\widetilde{\mu_n}*\mu_n\|_1\\ &=\|(\delta_s*\widetilde{\mu_n}-\widetilde{\mu_n})*\mu_n\|_1\\
    &\leq \|\delta_s*\widetilde{\mu_n}-\widetilde{\mu_n}\|_1.
\end{align*}
Now for any measurable set $E\subseteq G,$
\begin{align*}
    \delta_s*\widetilde{\mu_n}(E) &= \int_G\int_G \chi_E(gh)d\delta_s(g)d\widetilde{\mu_n}(h)\\
   &= \int_G\chi_E(sh)d\widetilde{\mu_n}(h)\\
    &=\int_G\chi_{s^{-1}E}(h^{-1})d\mu_n(h)\\
    &=\int_G\chi_{E^{-1}s}(h)d\mu_n(h)\\
    &= \frac{1}{m(F_n)}\int_G\chi_{E^{-1}s \cap F_n}(h)dm(h)\\
    &=\frac{m(E^{-1}s \cap F_n)}{m(F_n)}\\
    &=\frac{m(E \cap s{F_n}^{-1})}{m(F_n)}\\
    &=\frac{1}{m(F_n)}\int_E\chi_{sF_n^{-1}}(h)dm(h).
\end{align*}
The second last equality follows from unimodularity. Hence, $\frac{d(\delta_s*\widetilde{\mu_n})}{dm}=\frac{1}{m(F_n)}\chi_{sF_n^{-1}}.$
Consequently, \begin{align*}
    \|\delta_s*\widetilde{\mu_n}-\widetilde{\mu_n}\|_1 &=\frac{1}{m(F_n)}\|\chi_{sF_n^{-1}}-\chi_{F_n^{-1}}\|_1\\
    &=\frac{1}{m(F_n)}m(sF_n^{-1}\Delta F_n^{-1})\\
    &=\frac{1}{m(F_n)}m\bigl((F_ns^{-1}\Delta F_n)^{-1}\bigr)\\
    &=\frac{1}{m(F_n)}m(F_ns^{-1}\Delta F_n).
\end{align*}
In the last equality we used $m(E^{-1})=m(E), E\subseteq G$ due to unimodularity of $G.$ Hence by the F\o lner property of $(F_n)_n$ we conclude that $(\widetilde{\mu_n}*\mu_n)_n$ is a Reiter sequence.
 \end{proof}

\begin{thm}\label{main_theorem}
    Let G be a l.c.s.c. amenable unimodular group with an admissible F\o lner sequence $(F_n)$ and $ \CM$ be a finite von Neumann algebra with a f.n trace  $\tau$. Suppose $ (\CM, G, \alpha )$ is a $\tau$-preserving dynamical system. Let $x\in \CM \cap   L^2_{wm}(\CM,\tau)$ and $ \Gamma =( \pi, u, \phi) \in \S$, then $\A_n( x, \Gamma)$ converges  to $0$ in a.u.  
\end{thm}
   \begin{proof}

       Since $\mathcal{U}_d$ is a compact group, the Stone-Weierstrass theorem implies that every continuous function $\phi$ on $\mathcal{U}_d$ can be uniformly approximated by matrix coefficient functions of the form
       \[
u \mapsto \langle u\xi,\eta\rangle,
~~ \xi,\eta\in \mathbb{C}^d .
\]
Therefore, it suffices to prove the convergence of $\A_n(\cdot, \cdot)$ for functions of the form $\phi(u)=\langle u\xi,\eta\rangle .$
By Theorem~\ref{Banach_Principle}, it is enough to establish the convergence of $\A_n(x, \Gamma)$ for $ x\in \CM\cap L^2_{wm}(\CM,\tau).$
Let $x\in \CM\cap L^2_{wm}(\CM,\tau)$ and let $L=\{g_i:\, i\in\mathbb{N}\}$
be a countable dense subset of $G$. Fix $\varepsilon>0$. By Proposition~\ref{Mean_ergodic_thm}, for each $i\in\mathbb{N}$, there exists a projection $e_i\in \CP(\CM)$ such that $ \tau(e_i^\perp)<\frac{\varepsilon}{2^i}$
and
\[
\A_n\bigl(x^*\alpha_{g_i}(x)\bigr)e_i
\longrightarrow
\tau\bigl(x^*\alpha_{g_i}(x)\bigr)e_i
\]
in $\|\cdot\|$. 
Define $ e:=\bigwedge_{i\in\mathbb{N}} e_i .$
Then $ \tau(e^\perp)<\varepsilon,$
and
\[
\A_n\bigl(x^*\alpha_{g_i}(x)\bigr)e
\longrightarrow
\tau\bigl(x^*\alpha_{g_i}(x)\bigr)e
\]
in $\|\cdot\|$  and uniformly in for  all $i\in\mathbb{N}$.
Since the action $\alpha$ is continuous and $\{g_i\}_{i\in\mathbb{N}}$ is dense in $G$, it follows that
\[
\A_n\bigl(x^*\alpha_g(x)\bigr)e
\longrightarrow
\tau\bigl(x^*\alpha_g(x)\bigr)e, \text{ in } \norm{\cdot} \text{ for all } g \in G.
\]
Now consider the function $ f:G\to \CM\otimes M_d(\mathbb{C})$, 
defined by
\[
f(g)=\alpha_g(x)e\otimes \pi(g)u,
\]
where $u\in \mathcal{U}_d$. Clearly, $f$ is bounded.
We now compute the first term on the right-hand side of Ineq. \eqref{Vand_inequality} for this function. In this case, we obtain
      \begin{align*}
f(g)^*f(gh_1^{-1}h_2)=e\alpha_g(x^*\alpha_{h_1^{-1}h_2}(x))e\otimes u^*\pi({h_1^{-1}h_2 })u.
      \end{align*}
Consequently, we note that 
\begin{align}\label{inequality_A}
          \frac{1}{m(F_k)^2} \int_{F_k} \int_{F_k}
&\norm{ \left( \frac{1}{m(F_n)} \int_{F_n}f( g)^* ~ f( gh_{1}^{-1}h_2) dm(g) \right) }dm(h_1) dm(h_2)\notag \\
&=\frac{1}{m(F_k)^2} \int_{F_k} \int_{F_k}\norm{\frac{1}{m(F_n)}\int_{F_n}( e \alpha_g(x^*\alpha_{h_1^{-1}h_2}(x))  e\otimes u^*\pi(h_1^{-1}h_2 )u)}dm(h_1) dm(h_2)\notag \\
&\leq \frac{1}{m(F_k)^2} \int_{F_k} \int_{F_k}\norm{ e \A_n(x^*\alpha_{h_1^{-1}h_2}(x)) e }dm(h_1)\, dm(h_2)\notag \\
&\leq \frac{1}{m(F_k)^2} \int_{F_k} \int_{F_k}|\tau(x^* \alpha_{h_1^{-1}h_2}(x))|dm(h_1)\, dm(h_2)
     \end{align}
 Hence by Ineq.\eqref{Vand_inequality} and using the lemma \ref{lemma_Reiter} we have 
 \begin{align*}
    \A'_n(x):=\frac{1}{m(F_n)}\int_{F_n}\alpha_g(x)e\otimes\pi(g)u dm(g)
\end{align*}
converges to $0$ in $\CM\otimes M_d(\mathbb{C})$. Consider the linear map $id\otimes w_{\xi,\eta}: \CM \otimes M_d(\mathbb{C}) \to \CM$ defined as $id\otimes w_{\xi,\eta}(x\otimes u)=\langle u\xi,\eta\rangle x$. Since $id\otimes w_{\xi,\eta}$ is a bounded linear map, by the Bochner integral property, 
\begin{align*}
  id\otimes w_{\xi,\eta}(\A'_n(x))= \mathbb{A}_n(x,\Gamma)e
\end{align*}
converges to $0$ in $\CM$. This completes the proof.
   \end{proof} 

In the following remark, we discuss our results in the context of classical settings. 
\begin{rem}\label{class-comp}
In \cite{Pavel}, Pavel Zorin-Kranich proved the classical Wiener--Wintner theorem for amenable group actions with respect to \(C\)-tempered F\o lner sequences by extending Bourgain's Return Times Theorem to the setting of amenable groups.
\begin{thm}[Bourgain's Return Times Theorem]
Let $(X,\mathcal{B},\mu,T)$ be an ergodic measure-preserving system, and let
$f\in L^\infty(X)$. Then, for $\mu$-almost every $x\in X$, the sequence
\[
c_n=f(T^n x), \qquad n\geq 1,
\]
is a \emph{good universal weight} for the pointwise ergodic theorem. That is, for every measure-preserving system
$(Y,\mathcal{C},\nu,S)$ and every $g\in L^\infty(Y)$, the averages
\[
\frac{1}{N}\sum_{n=1}^{N} c_n\,g(S^n y)
=
\frac{1}{N}\sum_{n=1}^{N} f(T^n x)\,g(S^n y)
\]
converge as $N\to\infty$ for $\nu$-almost every $y\in Y$.
\end{thm}
To the best of our knowledge, no noncommutative analogue of Bourgain's Return Times Theorem is currently available, even for $\mathbb{Z}$-actions. Consequently, there is no corresponding version of the noncommutative analogue for actions of amenable groups. In his paper \cite{ElAbdalaoui2021}, Abdalaoui proved the Wiener-Wintner theorem, in the classical setting, for amenable groups along a F\o lner sequence that satisfies the Tempelman condition.

\medskip
\noindent
We would like to draw the reader's attention to the fact that our approach is partially inspired by that of M.~B\"ahring \cite{Bahring2019}. In his work, B\"ahring established the uniform Wiener--Wintner theorem for locally compact abelian groups in the classical setting by proving a finitary version of van der Corput's lemma and employing a suitable JdLG decomposition for locally compact abelian groups.

The main obstacle in extending the argument for more general amenable groups arises in the proof of Lemma~\ref{lemma_Reiter}.  In \cite{Bahring2019}, the corresponding result is established only for abelian groups. A careful survey of the classical ergodic theory indicates that, in order to prove Lemma~\ref{lemma_Reiter} for general amenable groups, one is naturally led to assume that the underlying F\o lner sequence satisfies the Tempelman condition.

\medskip
\noindent 
Indeed, we first observe that, by Theorem~\ref{existence_of_F'_n}, given an admissible Følner sequence $(F_n)$ satisfying the Tempelman condition, one can construct another sequence $(F_n')$ such that
\begin{align}
F_n' \subseteq F_n
\qquad\text{and}\qquad
\frac{m(F_n\setminus F_n')}{m(F_n)} \longrightarrow 0
\quad\text{as } n\to\infty,
\end{align}
with the additional properties that, for all sufficiently large $n\in\mathbb{N}$,
$(F_n')$ satisfies the Tempelman condition and
$(F_n'^{-1}F_n')$ is a Følner sequence.

Furthermore, for every $x\in\CM$ and $\Gamma\in\S$, we have
\begin{align*}
\A_n(x,\Gamma)
&=
\frac{m(F_n')}{m(F_n)}
\cdot
\frac{1}{m(F_n')}
\int_{F_n'}
\alpha_g(x)\phi(\pi(g)u)\,dm(g) \\
&\qquad
+
\frac{1}{m(F_n)}
\int_{F_n\setminus F_n'}
\alpha_g(x)\phi(\pi(g)u)\,dm(g).
\end{align*}
Since $ \frac{m(F_n\setminus F_n')}{m(F_n)}\longrightarrow 0
~\text{ and } ~
\|\alpha_g(x)\phi(\pi(g)u)\|
\leq \|x\|\,\|\phi\|,$
the second term converges to zero in norm. Consequently,
$\A_n(x,\Gamma)$ converges a.u.\ (respectively, b.a.u.) if and only if
\[
\frac{1}{m(F_n')}
\int_{F_n'}
\alpha_g(x)\phi(\pi(g)u)\,dm(g)
\]
converges a.u.\ (respectively, b.a.u.). Therefore, without loss of generality, we may assume throughout that $(F_n^{-1}F_n)$ is also a Følner sequence.

Then, following similar arguments available in the literature (see, e.g., \cite{ElAbdalaoui2021}, \cite{BeyerStroh2010}), one obtains an alternative proof of Theorem~\ref{main_theorem} for Følner sequences satisfying the Tempelman condition. Indeed, by Fubini's theorem and the change of variables $h_2 \mapsto h_1^{-1}h_2$, we have
\begin{align*}
\frac{1}{m(F_k)^2}\int_{F_k}\int_{F_k}
&|\tau(x^*\alpha_{h_1^{-1}h_2}(x))|
\,dm(h_1)\,dm(h_2)  \\
&=
\frac{1}{m(F_k)}
\int_{F_k^{-1}F_k}
\frac{m(F_k\cap F_kh^{-1})}{m(F_k)}
|\langle \alpha_h(x),x\rangle|
\,dm(h)  \\
&\le
\frac{C}{m(F_k^{-1}F_k)}
\int_{F_k^{-1}F_k}
|\langle \alpha_h(x),x\rangle|
\,dm(h),
\end{align*}
where the last inequality follows from the Tempelman condition
\[
m(F_k^{-1}F_k)\le C\,m(F_k).
\]
Since $x\in \mathcal{M}\cap L^2_{\mathrm{wm}}(\mathcal{M})$, letting $k\to\infty$ yields
\[
\frac{1}{m(F_k)^2}\int_{F_k}\int_{F_k}
|\tau(x^*\alpha_{h_1^{-1}h_2}(x))|
\,dm(h_1)\,dm(h_2)\longrightarrow 0.
\]
However, this approach is inherently restrictive, since a general locally compact amenable group need not admit a Tempelman F\o lner sequence.

\medskip
In the present work, we establish the Wiener-Wintner theorem for arbitrary unimodular l.c.s.c amenable groups. A major obstacle is that, in our noncommutative setting, no appropriate JdLG decomposition was previously available. To overcome this difficulty, we first develop a JdLG decomposition for amenable group actions on finite von Neumann algebras in the appropriate form. We then prove a noncommutative van der Corput lemma and combine it with an equivalent characterization of weak mixing along Reiter sequences to establish our main theorem.

Moreover,  we would also like to emphasize that Lemma~\ref{lemma_Reiter} enables a substantial strengthening of the main result of M.~B\"ahring~\cite{Bahring2019} on the uniform Wiener--Wintner theorem for unimodular locally compact amenable groups. To the best of our knowledge, this improvement is new even in the classical (commutative) setting.

\medskip
Finally, our results also solve the questions raised by Abdalaoui in  \cite[Question 2.12]{ElAbdalaoui2021} in the context of amenable groups, 
which we will quote verbatim as:
\emph{
Question 2.12  We ask on the possible extension of van der Corput inequality to the locally compact group and its application to obtain produce a direct
proof of the group version of Wiener-Wintner theorem.}

\end{rem}

\noindent\textbf{Acknowledgments.} 
First author P. Bikram gratefully acknowledges financial support provided by the Anusandhan National Research Foundation (ANRF), Government of India, under the : ANRF/ARGM/2025/001021/MTR  grant.

\providecommand{\bysame}{\leavevmode\hbox to3em{\hrulefill}\thinspace}
\providecommand{\MR}{\relax\ifhmode\unskip\space\fi MR }
\providecommand{\MRhref}[2]{%
\href{http://www.ams.org/mathscinet-getitem?mr=#1}{#2} }
\providecommand{\href}[2]{#2}

\end{document}